\documentclass{amsart}
\usepackage{amsmath}
\usepackage{amsfonts}
\usepackage{amssymb,amscd,amsthm}
\usepackage{geometry}
\usepackage{mathrsfs}
\usepackage[bookmarksnumbered, colorlinks, linktocpage, plainpages]{hyperref}
\usepackage{nameref}
\usepackage{wasysym,amssymb,eufrak,indentfirst,color}
\usepackage{ragged2e}
\usepackage{tikz}
\usepackage{mathtools}
\usepackage{appendix}
\usepackage{extarrows}
\usepackage{setspace}
\usepackage{cases}

\numberwithin{equation}{section}
\newcommand\asertion[1]{ssertion $({\mathrm{\romannumeral #1\relax}})$}

\newcommand\obj[1]{\text{obj}\,(\mathcal{#1})}
\newcommand{\ass}{associator}
\newcommand{\asselm}{associative element}

\newcommand{\almost}{para-}

\newcommand{\Hom}{\text{Hom}}
\newcommand{\End}{\text{End}}

\newcommand{\bfs}{Without loss of generality we can assume}

\newcommand{\spr}{\mathbb{R}}
\newcommand{\spo}{\mathbb{O}}
\newcommand{\lif}{\mathrm{lif}\,}
\newcommand{\ext}{\mathrm{ext}\,}

\newcommand\huaa[1]{\mathscr{A}(#1)}
\newcommand\hua[3]{\mathscr{#1}^{#2}(#3)}
\newcommand\pureim[1]{\mathit{Im}(#1)}
\newcommand{\algma}{\mathcal{U}}
\newcommand{\re}{\text{Re}\,}%

\newcommand\fx[2]{\left<#1,#2\right>}%

\def\O{\mathbb{O}}
\def\R{\mathbb{R}}

\newtheorem{mydef}{Definition}[section]
\newtheorem{rem}[mydef]{Remark}
\newtheorem{eg}[mydef]{Example}
\newtheorem{cor}[mydef]{Corollary}
\newtheorem{prop}[mydef]{Proposition}
\newtheorem{lemma}[mydef]{Lemma}
\newtheorem{thm}[mydef]{Theorem}

\begin{document}

\title[Non-associative categories ]{Non-associative categories  of octonionic bimodules}
\author{Qinghai Huo}
\email[Q.~Huo]{hqh86@mail.ustc.edu.cn}
\address{Department of Mathematics, University of Science and Technology of China, Hefei 230026, China}

\author{Guangbin Ren}
\email[G.~Ren]{rengb@ustc.edu.cn}
\address{Department of Mathematics, University of Science and Technology of China, Hefei 230026, China}

 \date{\today}


\thanks{This work was supported by the NNSF of China (12171448).}

	\begin{abstract} Category is put to work in the non-associative realm in the article. We focus on
  a typical example of non-associative category. Its objects are octonionic bimodules,   morphisms   are   octonionic  para-linear maps,   and compositions are non-associative in general.    The octonionic  para-linear map is the main  object of octonionic Hilbert  theory because of  the octonionic Riesz representation theorem.
An octonionic  para-linear   map $f$  is in general not  octonionic linear since it subjects to the rule
$$ \re\big(f(px)-pf(x)\big) =0.$$
 The   composition should be  modified as $$f\circledcirc g(x):=f(g(x))-\sum_{j=1}^7 e_j\re \Big(f\big(g(e_ix)\big)-f\big(e_ig(x)\big)\Big)$$
so that it  preserves the octonionic para-linearity.

 In this non-associative category,  we introduce   the   Hom    and   Tensor functors which constitute  an adjoint pair. We establish  the Yoneda lemma in terms of the new notion of weak functor. To define the exactness in a non-associative category, we introduce the notion of the enveloping category via a universal property.   This allows us to establish the exactness of the Hom functor and Tensor functor.
    	\end{abstract}

\maketitle

\tableofcontents

	\section{Introduction}

Category theory,  initiated by Eilenberg and MacLane  in 1945  \cite{MacLane1998Categories},  provides  deep insights and similarities between seemingly different areas of mathematics
\cite{Awodey_2006cat,MacLane1998Categories,Rotman2009homologicalalgebra,weibel1994homologicalgebra}.
However, category theory is exclusively restricted to the associative realm since	   compositions among morphisms    are always assumed to be   associative.

The purpose of this article is to bring the category into the non-associative realm by
providing  a typical example of  the non-associative category whose     compositions are  non-associative.  This category comes from  octonionic functional analysis.
In octonionic functional analysis, the  central objects are
octonionic  para-linear functionals (see  \cite{huoqinghai2021Riesz}).
This notion can be generalized to a more general setting, which leads to    a  specific  non-associative category. The objects of this category are octonionic bimodules,  the morphisms are octonionic para-linear maps, and the compositions are regular compositions that are non-associative.

We first recall the notion of para-linearity by focusing our attention on  octonionic Hilbert spaces.

\begin{mydef}[\cite{huoqinghai2021Riesz}]\label{def:ohil}
	Suppose that  $H$ is  a left   $\mathbb{O}$-module and also a real Hilbert space with a real-valued inner product $\langle \cdot,\cdot\rangle_{\mathbb  R}$. We call $H$ an \textbf{ $\mathbb{O}$-Hilbert space} if there exists an $\R$-bilinear map $$\left\langle\cdot,\cdot \right\rangle :H\times H \rightarrow \mathbb{O}$$  such that
	$$Re \left\langle\cdot,\cdot \right\rangle=\langle \cdot,\cdot\rangle_{\mathbb  R}$$
	and
	satisfying axioms
	\begin{enumerate}
		\item \textbf{($\mathbb{O}$-\almost linearity)} $\left\langle\cdot,u\right\rangle$ is left $\mathbb{O}$-\almost linear for all $u\in H$.
		\item \textbf{($\mathbb{O}$-hermiticity)} $\left\langle u ,v\right\rangle=\overline{\left\langle v ,u\right\rangle}$ for all $ u,v\in H$.
		\item \textbf{(Positivity)} $\left\langle u ,u\right\rangle\in \spr^+$  and $\left\langle u ,u\right\rangle=0$ if and only if $u=0$.
	\end{enumerate}
\end{mydef}

The definition of octonionic Hilbert spaces can trace back to  1964 by Goldstine and Horwitz \cite{goldstine1964hilbert}.
In fact, they put an extra axiom which states
\begin{equation*}\label{eq:def-redu-403}\left\langle pu,  u\right\rangle=p\left\langle u, u\right\rangle
\end{equation*}
for all $p\in \O$ and  $u\in H$. This axiom turns out to be  {\color{red}redundant}
\cite{huoqinghai2021Riesz}.

Definition \ref{def:ohil} provides a unified  viewpoint  for the theory of Hilbert spaces over normed division algebras.  In contrast to the full development of quaternionic Hilbert spaces \cite{horwitz1993QHilbertmod,razon1992Uniqueness,		razon1991projection,soffer1983quaternion,viswanath1971normal,Colombo2008funcalculus,ghiloni2013slicefct,colombo2011noncomfunctcalculus,Ghiloni2018semigp},
the research on octonionic   Hilbert spaces   is very rare.

The map $\left\langle\cdot,u\right\rangle$ induced by the $\O$-inner product
is a real linear map but not an $\O$-linear map. Even though this, it satisfies
$$\mbox{Re} \left\langle px,u\right\rangle = \mbox{Re} (p\left\langle x,u\right\rangle)$$ for any $p\in\O$.
This motivates us to introduce  a notion of $\mathbb{O}$-\almost linear function as
a real linear map $f:H\to\O$ for which
\begin{equation}\label{eq:def-redu-405}\re f(px)= \re (pf(x))
\end{equation}
for any $p\in \O$ and $x\in H$.
The importance of octonionic para-linear maps has been shown in
the octonionic Riesz representation theorem which  provides a bijection between
an  $\O$-Hilbert space  $H$  and
its  dual space  consisting  of $\O$-para-linear functions \cite{huoqinghai2021Riesz}.
We will further study   $\spo$-para-linear operators between
octonionic Hilbert spaces.
To this end, we need to 	
generalize the notion of para-linear functions to para-linear maps between $\O$-modules.

Let $M$ and $M'$ be two $\O$-bimodules.
We want to define the notion of left $\O$-para-linear maps from $M$ to $M'$ in terms of the definition equality \eqref{eq:def-redu-405}.
But in contrast to the case where $M'=\O$,  the real part operator in  (\ref{eq:def-redu-405}) is meaningless for  a general  $\O$-module $M'$.
Therefore,    we need to introduce a  real part structure in the target module $M'$, which is given by a projection map
\begin{equation}\label{eq:def of real part-411}
\mbox{Re}: M'\to \huaa{M'}.
\end{equation}
Here $\huaa{M'}$ is the set of associative elements of $M'$.
Notice that in the specific case  $M'=\O$  we have
\begin{equation}\label{eq:def of real part-410}
\re x=\frac{5}{12}x-\frac{1}{12}\sum_{i=1}^7 e_ixe_i.
\end{equation}
It turns out that the real part operator on  an $\O$-bimodule $M'$ also satisfies (\ref{eq:def of real part-410}). Similar results also hold in quaternionic case \cite{ng2007quaternionic}.

The set $\Hom_{\mathcal{LO}}(M,M')$  of all  left   $\spo$-\almost linear maps  can be endowed with a canonical  $\O$-bimodule structure. Its
$\O$-scalar multiplications are defined by
\begin{eqnarray*}
	(r\odot f)(x)&:=& f(xr)+A_r(x,f),
	\\
	(f\odot r)(x)&:=& f(xr)+A_r(x,f),
\end{eqnarray*}
where $$A_r(x,f):=f(rx)-rf(x).$$
We find that the  real part of the bimodule $\Hom_{\mathcal{LO}}(M,M')$ is
$$\re \big(\Hom_{\mathcal{LO}}(M,M')\big)=\Hom_\O(M,M').$$
This indicates that it is too small to study the $\O$-linear maps in octonionic functional analysis.

Note that the para-linear maps are not preserved under ordinary composition.
We then introduce a new  notion of regular composition which preserves octonionic  para-linear morphisms.
More precisely,
let $M, M', M''$ be $\O$-bimodules and	let  $f\in\Hom_{\mathcal{LO}}(M',M'')$ and  $g\in  \Hom_{\mathcal{LO}}(M,M')$.
We define
$$(f\circledcirc g)(x):=f(g(x))-\sum_{j=1}^7 e_j\re \Big(f\big(g(e_ix)\big)-f\big(e_ig(x)\big)\Big).$$
Then we have
$$f\circledcirc  g\in \Hom_{\mathcal{LO}}(M,M'').$$

The above definition originates from   the specific case  of the multiplication of octonionic matrices.
Notice that
\begin{eqnarray}\label{eq:reg-comp-290}\Hom_{\mathcal{LO}}(\O,\O)=\{R_p: p\in\O\},\end{eqnarray}
where $R_p$ is a right multiplication operator on $\O$ defined by
\begin{eqnarray*}\label{eq:reg-comp-292}
	R_p x=xp.
\end{eqnarray*}
It turns out that
\begin{eqnarray}\label{eq:reg-comp-295}
R_p\circledcirc R_q=R_{pq}
\end{eqnarray}
so that $\Hom_{\mathcal{LO}}(\O,\O)$ is isomorphic to $\O$ as an algebra.
The multiplications of general matrices  have the  similar properties.
Notice  that the regular composition is generally non-associative. This gives rise to our specific non-associative category. We shall denote by $\mathcal{LO}$  the typical non-associative   category,
whose objects are  octonionic bimodules,    whose morphisms are  left $\O$-para-linear maps, and  whose compositions are left regular compositions.
The category $\mathcal{RO}$ of right  para-linear maps  can be defined similarly.

Now we move on to study  functors 	between these two  non-associative   categories. We shall introduce several functors, including the  conjugate,  Hom,  and tensor functors.

Firstly, we define the   conjugate functor.
For any $\O$-bimodule, we have an induced bimodule, denoted by $M^{C}$,  whose  bimodule structure  is defined as:
$$p\cdot x=x\overline{p},\quad x\cdot p=\overline{p}x.$$
This leads to a functor,   called the    conjugate  functor,
$$C:\mathcal{LO}\to\mathcal{RO},$$
which is an isomorphism of categories.
The effect of  the conjugate functor is that it can translate  results from    $\mathcal{LO}$ to   $\mathcal{RO}$ and vise versa.

Next, we define $\Hom$ functors between categories $\mathcal{LO}$  and $\mathcal{RO}$. They will   play a central role in the study of dual operators in octonionic functional analysis.

\begin{thm}
	For any $\O$-bimodule $M$, we have a covariant Hom functor $$\Hom_\mathcal{LO}(M,-):\mathcal{LO}\to\mathcal{LO}
	$$ and a contravariant Hom functor $$\Hom_\mathcal{LO}(-,M):\mathcal{LO}\to\mathcal{RO}.$$
	Moreover, both functors are $\O$-linear.
\end{thm}
\noindent

Finally, {we consider  the  tensor functors. We define  the  octonionic tensor product of two $\O$-bimodules $M$, $M'$ as
	$$M\otimes_\O M':=(\re M\otimes_\R \re M')\otimes\O.$$
	We then have a natural tensor functor:
	\begin{align*}
	M\otimes^{ll}_\O-:\qquad\mathcal{LO}\qquad&\xlongrightarrow[]{\hskip1cm} \qquad\mathcal{LO}\\
	X\qquad&\shortmid\!\xlongrightarrow[]{\hskip1cm} \qquad M\otimes_\O X\\
	X\stackrel{f}{\longrightarrow}X'\qquad&\shortmid\!\xlongrightarrow[]{\hskip1cm} \qquad l\text{-}\ext f_M.
	\end{align*}
	Here the map $$f_M:\re M\otimes_\R \re X\to \re M\otimes_\R \re X'$$ is defined  by
	$$f_M(m\otimes x):=m\otimes f(x)$$ for all $m\in \re M$ and $x\in \re X$. The map $l\text{-}\ext f_M$ is  a canonical  \almost linear extension  map of $f_M$.
	Similarly, 	we can also define functors
	$$M\otimes^{lr}_\O-, \qquad\quad  M\otimes^{rr}_\O-, \qquad M\otimes^{rl}_\O-.$$
	It turns out that the two functors $$(M\otimes^{ll}_\O-, \ \Hom_\mathcal{LO}(M,-))$$ constitute  an \textbf{adjoint pair}.}

The above preparation allows us to  study further   exactness  for non-associative categories.
It is hard to consider  exactness
in the non-associativity setting. To overcome the difficulty, we  associate   the non-associative category with an associative category. This associative category is constructed due to  a universal property  and is called   the enveloping category.
The exactness in a non-associative category can thus be defined as the exactness of its enveloping category.

Studying the exactness, we need to introduce the notion of natural associated transformations.   Through a natural associated isomorphism, a functor
between non-associative categories   naturally corresponds to an associated  functor between their enveloping categories.

For example, we can provide  the natural associated functors of the three functors $\Hom_\mathcal{LO}(-,M)$, $\Hom_\mathcal{LO}(M,-)$,  and $	M\otimes^{ll}_\O-$.
It turns out that
the natural associated functors of $\Hom_\mathcal{LO}(-,M)$
is the usual Hom functor $$\Hom_\R(-, \re M):\O\text{-}\textbf{Mod}_\R\stackrel{}{\xlongrightarrow[]{\hskip1cm}}\O\text{-}\textbf{Mod}_\R,$$
{where $\O\text{-}\textbf{Mod}_\R$ stands for the   associative category,  whose objects are $\O$-bimodules,  morphisms are $\R$-linear maps,  and compositions are usual compositions.}
The natural associated functor of $\Hom_\mathcal{LO}(M,-)$ is the usual Hom functor $$\Hom_\R(\re M,-):\O\text{-}\textbf{Mod}_\R\stackrel{}{\xlongrightarrow[]{\hskip1cm}}\O\text{-}\textbf{Mod}_\R,$$
and the natural associated functor of $	M\otimes^{ll}_\O-$   is the usual tensor functor
$$\re M\otimes_\R-:\O\text{-}\textbf{Mod}_\R\stackrel{}{\xlongrightarrow[]{\hskip1cm}}\O\text{-}\textbf{Mod}_\R.$$
These    yield  the exactness of Hom functor and tensor functor in  $\mathcal{LO}$.

In the classical theory of categories, the Yoneda lemma
is arguably  the most important result.
We can extend this result to non-associative categories.
It  gives rise to  a new notion called   weak functors. It turns out that  the usual Hom functor  from  a non-associative category to the $\text{\bf{Sets}}$ category  is merely a weak functor other than a functor.


The organization of this article is as follows.
In Section \ref{sec:prel}, we  give a brief overview of the algebra of octonions, and   review some definitions and basic properties on the theory of  octonionic modules;  more details can be found in \cite{huo2021leftmod,Schafer1952repaltalg,Shestakov2016bimod}.
In Section \ref{sec:almost}, we introduce the notion of  $\spo$-\almost linear maps and give a detailed discussion on  this concept.
In Section \ref{sec:cat},   we adapt the method of category  to study $\O$-para-linear maps and initiate the study of   non-associative categories. 	

\hskip1cm

	\section{Preliminaries}\label{sec:prel}
In this section, we review some basic properties of octonions and octonionic modules.

\subsection{Octonions}\label{sec:O}
The algebra $\spo$ of octonions     is a  non-associative, non-commutative, normed division algebra over   $\spr$. As a real vector space, it has  a basis  \cite{baez2002octonions}
$$e_0=1, \quad e_1, \quad \dots, \quad e_7, $$ subject to the multiplication rules
$$e_ie_j+e_je_i=-2\delta_{ij},\quad i,j=1,\dots,7.$$
An octonion can be written as $$x=x_0+\sum_{i=1}^7x_ie_i,\quad x_i\in\spr.$$
Its conjugate   is defined as $$\overline{x}=x_0-\sum_{i=1}^7x_ie_i.$$
The norm of $x$ can be expressed as $$|x|=\sqrt{x\overline{x}}\in \spr; $$ and the real part of $x$ is given by $$\re{x}=x_0=\frac{1}{2}(x+\overline{x}).$$

Following the $\epsilon$-notation from
\cite{bryant2003some},
the multiplication rules can be expressed as
\begin{align}
&e_ie_j=\epsilon_{ijk}e_k-\delta_{ij}.
\end{align}
{The symbol $\epsilon$  is skew-symmetric in  indices. In fact, if we  write  $e^{i j k}$ for the wedge product  $e^{i} \wedge e^{j} \wedge e^{k}$ in  $\Lambda^{3}\left(\left(\mathbb{R}^{7}\right)^{*}\right)$, then we obtain a $3$-form:
	\begin{align*}
	\phi &=\frac{1}{6} \varepsilon_{i j k} e^{ijk} .
	\end{align*}
	More precisely, we have
	\begin{align*}
	\phi & = e^{123}+e^{145}+e^{167}+e^{246}-e^{257}-e^{347}-e^{356}.
	\end{align*}
}

 \subsection{$\O$-modules}
There are abundant results on  octonionic bimodules, or more generally, alternative modules. Already  in 1952, Schafer \cite{Schafer1952repaltalg} gave the birepresentations of alternative algebras. Subsequently, Jacobson \cite{jacobson1954structure} determined the irreducible representations for finite
dimensional semi-simple alternative algebras.   A more general study for alternative bimodules is given in \cite{Shestakov2016bimod}. Recently, only for  the  octonionic case, we \cite{huo2021leftmod} introduce the new notions of associative elements and conjugate associative elements, which   give rise to   a new description of the structure of $\O$-modules.

We now  recall  some basic notations and    results on $\O$-modules.
A  real vector space $M$ is called  a left  $\O$-module if it is endowed with an $\O$-scalar multiplication such that    the \textbf{left associator} is    left  alternative.
That is, for all  $p,q\in\O,  \ m\in M$,
\begin{eqnarray}\label{eq:left ass}
[p,q,m]=-[q,p,m],
\end{eqnarray}
where the left associator is defined by
$$[p,q,m]=(pq)m-p(qm).$$
It is easy to check that \eqref{eq:left ass} is  equivalent to
$$r(rm)=r^2m \text{\ \ \  for all \ \ }r\in \O,\, \ m\in M.$$

We recall a useful
identity which holds in any left $\O$-module $M$.
\begin{equation}\label{eq:[p,q,r]m+p[q,r,m]=[pq,r,m]-[p,qr,m]+[p,q,rm]}
[p,q,r]m+p[q,r,m]=[pq,r,m]-[p,qr,m]+[p,q,rm]
\end{equation}
for all $p,q,r\in \O$ and $m\in M$.
We shall generalize this identity in  many different settings.

The \textbf{nucleus}  of a left $\O$-module $M$ is defined by   $$\huaa{M}:=\{m\in M\mid [p,q,m]=0,\text { for all } p,q \in \O\}.$$
Any element in the nucleus $\huaa{M} $ is called an \textbf{associative element}.
We denote by $\hua{A}{-}{M}$  the set of all \textbf{conjugate associative elements}
$$\hua{A}{-}{M}:=\{m\in M\mid (pq)m=q(pm)\text{ for all } p,q\in \spo\}.$$
The irreducible left $\O$-modules are already known to be isomorphic to the regular  or conjugate regular modules. We refer to Chapter 11 of the monograph  \cite{zhevlakov1982Rings} for detailed results. We \cite{huo2021leftmod}   describe the structure of a left $\O$-module in terms of associative elements and conjugate associative elements  as follows
\begin{thm}[\cite{huo2021leftmod}]\label{thm:M=OA+OA^-}
	Let $M$ be a left $\O$-module. Then  $$M=\spo\huaa{M}\oplus {\spo}\hua{A}{-}{M}.$$
\end{thm}

We now recall the definition of $\O$-bimodules. Since $\O$ is a  special alternative algebra, the notion of an $\O$-bimodule  is just an alternative bimodule over $\O$ \cite{Schafer1952repaltalg,jacobson1954structure}. Here we state the definition as follows.
\begin{mydef}
	A left  $\O$-module $M$  is called an  \textbf{$\O$-bimodule}  if  there is a right $\R$-bilinear multiplications  such that 	for any $p,q\in \O$ and $m\in M$,
	$$[p,q,m]=[m,p,q]=[q,m,p].$$
	Here we denote
	$$[q,m,p]:=(qm)p-q(mp) $$
	and denote the similar notation for $[m,p,q]$.
 	
\end{mydef}

We denote by $\text{Reg}\,{\O}$, or just $\O$ if there is no confusion, the regular bimodule  with the multiplication given by the product
in $\O$. Clearly,   the $\O$-bimodule $\text{Reg}\,{\O}$ is
irreducible. Moreover, it is the only
irreducible $\O$-bimodule; see  the results by Schafer \cite{Schafer1952repaltalg} and   Jacobson \cite{jacobson1954structure}.

\section{$\spo$-\almost linearity}\label{sec:almost}
Recently, we \cite{huoqinghai2021Riesz} introduce a notion of $\spo$-\almost linear functional. In fact, we can generalize this notion to a more general case of  $\spo$-\almost linear maps between $\spo$-modules.
The  $\spo$-\almost linearity  plays a central role in the study of octonionic functional analysis.  We shall give  a complete discussion about this concept in this section.
\subsection{Real part operator on $\O$-bimodules}

To generalize the notion of \almost linear functions to  \almost linear maps, we need to generalize the notion of the  real part structure of $\O$ to general $\O$-bimodules.

Let $M$ be an $\O$-bimodule. Since the only irreducible $\O$-bimodule is the regular bimodule   $\text{Reg}\,{\O}$, it follows from
Theorem \ref{thm:M=OA+OA^-} that
\begin{eqnarray}\label{eq: M=O ReM}
M=\O\huaa{M}.
\end{eqnarray}
One can easily check that $$e_i\huaa{M}\cap e_j\huaa{M}=\{0\}$$ for each $i\neq j\in\{0,\dots,7\}$.
{Indeed, supposing  on the contrary, we can find $$e_ix=e_jy$$ for some non-zero elements $x,y\in \huaa{M}$. Then $x=(\overline{e_i}e_j)y$. It follows that
$$[p,q,x]=[p,q,\overline{e_i}e_j]y=0$$ for arbitrary $p,q\in \O$.
This forces that $y=0$, a contradiction.}

Hence we have a direct sum decomposition of $M$
\begin{eqnarray}\label{eq:M=sum ei ReM}
M= \bigoplus_{i=0}^7e_i\huaa{ M}.
\end{eqnarray}
We thus define the real part operator as the projective operator
$$\re: M\to \huaa{M}.$$

The right multiplication in an $\O$-bimodule is uniquely determined by its left multiplication, which is completely different from the quaternionic case.
We denote	the \textbf{commutative center} of  $M$ by: $$\hua{Z}{}{M}:=\{x\in M\mid px=xp, \text{ for all }p\in \O\}.$$
Then we have	\begin{eqnarray}\label{eq:reM=huaaM=ZM}
\re M=\huaa{M}=\hua{Z}{}{M}.
\end{eqnarray}

We collect some properties of the real part operator as follows.
\begin{thm}\label{lem:real part on good bimod }
	Let  $M$ be an $\O$-bimodule.
	We have a concrete expression of the real part operator in terms of scalar multiplications
	\begin{equation}\label{eq:def of real part}
	\re x=\frac{5}{12}x-\frac{1}{12}\sum_{i=1}^7 e_ixe_i.
	\end{equation}
	Moreover, for all $x\in M$, we have
	\begin{enumerate}
		\item 	$\re^2x=\re x$.
		
		\item  For all   $p,q\in \O$,  \begin{eqnarray}
		\re (p\re (x))&=&(\re p)\re(x),\label{eq:real part on good bimod }\\
		\re [p,q,x]&=&\re [p,x]=0.\label{prop:re part}
		\end{eqnarray}
		\item  $x=\sum_{i=0}^7 e_i\re(\overline{e_i}x)$.
 	\end{enumerate}
	
\end{thm}
\begin{proof}
	We  prove \eqref{eq:def of real part}.
	One can check directly 	that \eqref{eq:def of real part} holds for the case $M=\O$.   The general case then follows. Indeed, for any $x\in M$, write $x=\sum_{j=0}^7e_jx_j$ with $x_j\in \huaa{M}$. Then according to \eqref{eq:reM=huaaM=ZM}, we have
	\begin{align*}
	\frac{5}{12}x-\frac{1}{12}\sum_{i=1}^7 e_ixe_i&=\frac{5}{12}\sum_{j=0}^7e_jx_j-\frac{1}{12}\sum_{i=1}^7 e_i\left(\sum_{j=0}^7e_jx_j\right)e_i\\
	&=\sum_{j=0}^7\left(\frac{5}{12}e_j-\frac{1}{12}\sum_{i=1}^7 e_ie_je_i\right ) x_j\\
	&=\sum_{j=0}^7\re(e_j)x_j\\
	&=x_0.
	\end{align*}
	This proves \eqref{eq:def of real part}.
	
	The  proof of a\asertion{1}  is trivial. We prove \eqref{eq:real part on good bimod }. \bfs\ $x\in \huaa{M}$.
	It follows from a\asertion{1} that
	\begin{align*}
	\re(px)&=\frac{5}{12}px-\frac{1}{12}\sum_{i=1}^7 e_i(px)e_i\\
	&=\frac{5}{12}px-\frac{1}{12}\sum_{i=1}^7 (e_ipe_i)x\\
	&=(\re p)x.
	\end{align*}
	Then \eqref{prop:re part} follows from \eqref{eq:real part on good bimod } directly.

	Let $x=\sum_{i=0}^7e_ix_i$. Then it follows from \eqref{eq:real part on good bimod } that
	$$\re(\overline{e_i} x)=\re\sum_{i=0}^7(e_j\overline{e_i}) x_j=x_i.$$
	This proves a\asertion{3}.
\end{proof}

\begin{rem}
	Recall that in a  quaternionic bimodule, the real part operator is defined by
	$$\re x=\frac{1}{4}\sum _{e\in B}\overline{e}xe \quad (x\in X).$$
	Where $B := \{1,i, j, k\}$ is a basis of the quaternions $\mathbb{H}$ (\cite{ng2007quaternionic}). And  the polarization identity holds:
	\begin{eqnarray}
	x=\sum _{e\in B}e\re(\overline{e} x).
	\end{eqnarray}
	\eqref{eq:def of real part} is a similar result  in octonionic case.
	
 \end{rem}

Let $M$, $M'$ be two $\O$-bimodules.
A map $f:M\to M'$ is called \textbf{$\O$-linear} if $$  f(px)=pf(x), \quad f(xp)=f(x)p$$
for all $x\in M, \ p\in\O$.	 Note that the right multiplication in an $\O$-bimodule is uniquely determined by its left multiplication. We conclude that
\begin{equation}\label{eq:HomO(M,M')=l-Hom(M,M')}
\begin{split}
\Hom_\O(M,M')&=\{f\in \Hom_\R(M,M')\mid f(px)=pf(x) \text{ for all } x\in M, \ p\in\O \}\\
&=\{f\in \Hom_\R(M,M')\mid f(px)=pf(x) \text{ for all } x\in M, \ p\in\O \}.
\end{split}
\end{equation}

It follows from  \eqref {eq:def of real part} that any $\O$-linear map commutes with the  real part operator.
\begin{lemma}\label{lem:f re=re f}
	Let $M$, $M'$ be two $\O$-bimodules and $f\in \Hom_\O(M,M')$.
	Then $$f(\re x)=\re f(x)$$ for all $x\in M$.
\end{lemma}

\subsection{Second associator}\label{sec:second ass}
To study the  \almost linear map, we need  to introduce the notion of the second associators related to a linear map  between $\O$-modules and establish some identities associated with this notion.

Let $M$, $M'$ be two left $\O$-modules. For any  $f\in \Hom_\R(M,M')$, $p\in \O$ and $x\in M$, we denote
$$A_p(x,f):=f(px)-pf(x) $$ and call it the \textbf{second  left associator}  of $f$.
Similarly we can introduce the \textbf{second right associator} between right  $\O$-modules via  $$B_p(f,x):=f(x)p-f(xp).$$

By definition, we can generalize identity \eqref{eq:[p,q,r]m+p[q,r,m]=[pq,r,m]-[p,qr,m]+[p,q,rm]} as follows.

\begin{lemma}\label{lem: associator prop1 algm}
	Let $M$, $M'$ be two  $\O$-modules and $f\in \Hom_\R(M,M')$. Assume  that $\alpha \in \R$,  $p,q\in \O$ and $x\in M$. 	The following statements hold.	
	\begin{enumerate}
		\item If $M$, $M'$ are two left $\O$-modules, then
		\begin{equation}\label{eq: associator 1 id}
		A_{pq}(x,f)=f([p,q,x])-[p,q,f(x)]+pA_q(x,f)+A_p(qx,f).
		\end{equation}
		\item If  $M$, $M'$ are two right $\O$-modules,  then
		\begin{equation}\label{eq:associator 2 id:Bpq(f,x)}
		B_{pq}(f,x)=f([x,p,q])-[f(x),p,q]+B_p(f,x)q+B_q(f,xp).
		\end{equation}
		\item If $M$, $M'$ are two  $\O$-modules, then $$A_\alpha(x,f)=B_\alpha(f,x)=0.$$
	\end{enumerate}	
\end{lemma}
\begin{proof} A direct calculation shows that 		\begin{align*}
	A_{pq}(x,f)&=f((pq)x)	-(pq)f(x)\\
	&=f(p(qx)+[p,q,x])	-(pq)f(x)\\
	&=pf(qx)+A_p(qx,f)+f([p,q,x])-(pq)f(x)\\
	&=p(qf(x))+pA_q(x,f)+A_p(qx,f)+f([p,q,x])-(pq)f(x)\\
	&=f([p,q,x])-[p,q,f(x)]+pA_q(x,f)+A_p(qx,f).
	\end{align*}
	This proves a\asertion{1}. A\asertion{2} can be proved  similarly. A\asertion{3} is trivial.
\end{proof}
\begin{rem} Let $f$ be a map between two  left (right) $\O$-modules.
	We denote the
	evaluation   of $f$  on an element  $x$ by
	\begin{eqnarray}\label{eq:notation f(x)}
	\fx{x}{f}:=f(x)\quad (\text{resp. }\fx{f}{x}:=f(x)).
	\end{eqnarray}
	This notation is reminiscent of the inner product on a Hilbert space.

	If   we write  
	$$A_p(x,f)=\fx{px}{f}-p\fx{x}{f}=:[p,x,f],$$
	as we have denoted in \cite{huoqinghai2021Riesz},  then  identity \eqref{eq: associator 1 id} can be rephrased as
	$$p[q,x,f]+\fx{[p,q,x]}{f}=[pq,x,f]-[p,qx,f]+[p,q,\fx{x}{f}],$$
	which is an analogue of identity \eqref{eq:[p,q,r]m+p[q,r,m]=[pq,r,m]-[p,qr,m]+[p,q,rm]}.		
	
	Similarly, we denote 
	$$B_p(f,x)=\fx{f}{x}p-\fx{f}{xp}=:[f,x,p].$$ Then   identity \eqref{eq:associator 2 id:Bpq(f,x)} can be interpreted as an identity   similar  to  identity \eqref{eq:[p,q,r]m+p[q,r,m]=[pq,r,m]-[p,qr,m]+[p,q,rm]}.

	Although such notations
	provide   a neat mnemonic, 		
	we shall still keep  the  notations $A_p(x,f)$ and $B_p(f,x)$  throughout in order to avoid any confusions. The notation given in \eqref{eq:notation f(x)} will be used many times in the sequel.
\end{rem}

 The following   identities related to   second associators are crucial in applications.

\begin{prop}\label{prop:left,left associator prop}
	Let $M,\ M'$ be two  left $\spo$-modules. Then for all $p\in \spo,\ x\in M$, and $f\in \Hom_{\spr}(M,M')$, the following identities hold:
	\begin{enumerate}
		\item $A_{\overline{p}}(x,f)=-A_p(x,f)$;
		\item $A_{p^2}(x,f)=pA_p(x,f)+A_p(px,f)$;
		\item $A_p(\overline{p}x,f)=pA_p(x,f),\qquad A_p({p}x,f)=\overline{p}A_p(x,f)$;
		\item $A_p(\overline{p}^kx,f)=p^kA_p(x,f), k\in \mathbb{N}$.
		\item $A_{p^{k+1}}(x,f)=pA_{p^k}(x,f)+\overline{p}^kA_p(x,f), k\in \mathbb{N}$.
	\end{enumerate}
	Similar results hold for the   second right associators.	
\end{prop}
\begin{proof}
	A\asertion{1} follows from    a\asertion{3} of Lemma \ref{lem: associator prop1 algm}.  A\asertion{2} follows from
	identity \eqref{eq: associator 1 id} with $q$ replaced by $p$.  A\asertion{3} follows also from  identity \eqref{eq: associator 1 id} with $q$ replaced by $\overline{p}$.
	
	A\asertion{4} can be proved by induction on $k$. A\asertion{5} follows by induction on $k$ and using a\asertion{4}.
\end{proof}

To show the $\O$-module structures of the set of \almost linear maps,
we shall define several kinds of  octonionic scalar multiplications on $\Hom_{\R}(M,M')$ according to the types of $\O$-modules $M$ and $M'$.  The notation  \eqref{eq:notation f(x)} will be used in the following proposition
for memory convenience.

\begin{prop}\label{thm:mod strc over general mod}
	Let $M,M'$ be two $\O$-modules, $f\in \Hom_{\R}(M,M')$ and  $r,p,q\in \O$.

	\begin{enumerate}
		\item If $M$ is left module and $M'$ is bimodule, we define
		\begin{eqnarray}\label{eqdef:<x,fr>}
		\left<x,f\odot r\right>:=\left<x,f\right>r-A_r(x,f).
		\end{eqnarray}
		Then  we have
		\begin{align}\label{eq:algmod A_p(x,fr)=[p,f(x),r]+A_p(x,f)r-A_r(px,f)+pA_r(x,f)}
		A_p(x,f\odot r)&=[p,f(x),r]+A_p(x,f)r-A_r(px,f)+pA_r(x,f),
		\\  [f,p,q](x)&=[f(x),p,q]-A_p(x,f)q-A_q(x,f\odot p)+A_{pq}(x,f).\label{eq:alm left bi[f,p,q](x)=[f(x),p,q]-Ap(x,f)q-Aq(x,fp)+A pq(x,f)}
		\end{align}

		\item If $M$ is bimodule and $M'$ is  left module,  we define
		\begin{eqnarray}\label{eqdef:<x,rf>}
		\left<x,r\odot f\right>:=\left<xr,f\right>+A_r(x,f).
		\end{eqnarray}
		Then  we have
		\begin{align}			A_p(x,r\odot f)&=f([p,x,r])+A_p(xr,f)+A_r(px,f)-pA_r(x,f),\label{eq:Ap(x,rf)bileft}
		\\ [p,q,f](x)&=A_{pq}(x,f)-A_q(xp,f)-A_p(x,q\odot f)-f([x,p,q]).\label{eq:alm bi left[pqf]x in }
		\end{align}
		\item If $M$ is right module and $M'$ is  bimodule, we define
		\begin{eqnarray}\label{eqdef:<rf,x>}
		\left<r\odot f,x\right>:=r\left<f,x\right>+B_r(f,x).
		\end{eqnarray}
		Then  we have
		\begin{align}			B_p(r\odot f,x)&=[r,f(x),p]+B_r(f,xp)+rB_p(f,x)-B_r(f,xp),\label{eq:Bp(rf,x)rightbi}
		\\ [p,q,f](x)&=B_{pq}(f,x)-pB_q(f,x)-B_p(q\odot f,x)+[p,q,f(x)].\label{eq:alm right bimod[pqf]x in }
		\end{align}
		\item If $M$ is  bimodule and $M'$ is  right module,  we define
		\begin{eqnarray}\label{eqdef:<fr,x>}
		\left<f\odot r,x\right>:=\left<f,rx\right>-B_r(f,x).	
		\end{eqnarray}
		Then  we have 	
		\begin{align}			B_p(f\odot r,x)&=f[r,x,p]-B_r(f,x)p+B_p(f,rx)+B_r(f,xp),\label{eq:Bp(fr,x)biright}
		\\ [f,p,q](x)&=B_{pq}(f,x)-B_q(f\odot p,x)-B_p(f,qx)+f([p,q,x]).\label{eq:alm bi right[pqf]x in  mod}
		\end{align}
		\item If $M$ and $M'$ are two bimodules, then  with the left or right  multiplication defined by \eqref{eqdef:<x,fr>} and \eqref{eqdef:<x,rf>} we have
		\begin{align}
		[q,f,p](x)
				&=A_q(x,f)p-A_p(x,q\odot f)+A_p(xq,f)-A_q(x,f\odot p)\label{eq:bi,bi [qfp]=Aq(xf)p-A_p(x,qf)+A_p(xq,f)-A_q(x,fp)}.
		\end{align}
		Similarly,  with   the left or right  multiplication defined by \eqref{eqdef:<rf,x>} and \eqref{eqdef:<fr,x>}, we have
		\begin{align}
		[q,f,p](x)&=qB_p(f,x)-B_p(q\odot f,x)+B_q(f,px)-B_q(f\odot p,x).
		\end{align}
	\end{enumerate}
	
\end{prop}

 \begin{proof} The results follows   by   straightforward calculations. We prove a\asertion{1}  for example.  By definition, we have
	\begin{align*}
	A_p(x,f\odot r)&=\left<px,f\odot r\right>-p\fx{x}{f\odot r}\\
	&=\fx{px}{f}r-A_r(px,f)-p\left(f(x)r-A_r(x,f)\right)\\
	&=\left(pf(x)+A_p(x,f)\right)r-A_r(px,f)-p\left(f(x)r-A_r(x,f)\right)\\
	&=[p,f(x),r]+A_p(x,f)r-A_r(px,f)+pA_r(x,f).
	\end{align*}
	The rest proof  runs in the same manner and is omitted.
\end{proof}

\begin{rem}		
	Although we have defined the  octonionic scalar multiplication on $\Hom_{\R}(M,M')$, the set $\Hom_{\R}(M,M')$ is   not an  $\O$-module in general.   Indeed, by definition  \eqref{eqdef:<x,fr>},
	it may  hold
	$$f\odot(r^2)\neq(f\odot r)\odot r$$
	for some $r\in \O$ and  $f\in \Hom_{\R}(M,M')$.

	The identities in {Proposition} \ref{thm:mod strc over general mod} are all akin to  identity \eqref{eq:[p,q,r]m+p[q,r,m]=[pq,r,m]-[p,qr,m]+[p,q,rm]}.
	{Those identities about associators together  with the real part structure of modules   play a crucial role in the sequel. }
\end{rem}

\subsection{$\O$-\almost linear map}\label{subsec:almost linear}
In this subsection, we shall generalize the notion of \almost linear functions developed in our previous work \cite{huoqinghai2021Riesz} to arbitrary \almost linear maps between $\O$-modules.

\textbf{Convention: \quad  The target module is always assumed to be an $\O$-bimodule from now on, except where otherwise specified.}

\begin{mydef}\label{def:\almost linear}
	Let $M$ be a  left $\spo$-module and $M'$ an $\O$-bimodule. A map $f\in \Hom_\spr(M,M')$ is called \textbf{(left)  $\spo$-\almost linear} if
	\begin{eqnarray}\label{eq:re Ap(xf)=0}
	\re A_p(x,f)=0
	\end{eqnarray}
	for all $p\in \O$ and  $x\in M$.		
\end{mydef}

Similarly,  we can also define the notion of  right  $\spo$-\almost linear map with the  left \ass\ $A_p(x,f)$ replaced by the right associator $B_p(f,x)$.

We denote by $$\Hom_{\mathcal{LO}}(M,M')$$ the set of (left)  $\spo$-\almost linear maps and denote   $$\End_{\mathcal{LO}}(M):=\Hom_{\mathcal{LO}}(M,M).$$  In order to avoid any confusion, we use systematically the subscripts $\mathcal{LO}$ and  $\mathcal{RO}$   to indicate that the module under consideration is a left or a right module. But we may omit the subscripts when the module under consideration is understood well.

For  $f\in \Hom_\mathcal{LO}(M,M')$,	we shall only use the notation $\fx{x}{f}$ to  represent  the value $f(x)$.  Similarly, only the notation  $\fx{f}{x}$  is used for $f\in\Hom_\mathcal{RO}(M,M')$. 	We shall state our results for  left  $\spo$-modules, while   the similar results for  right $\spo$-modules  shall be collected  in   subsection \ref{sec:right O mod}.

Let $M$ be a left $\O$-module and $M'$ an $\O$-bimodule.
For   simplicity, we  denote  $$f_{\R}(x)=\re f(x)$$ 	
for any $x\in M$ and  $f\in \Hom_{\R}(M,M')$.
In view of  \eqref{eq:M=sum ei ReM},
$f$ can be written as   $$f(x)=f_{\R}(x)+\sum_{i=1}^7
e_if_i(x),$$
where  $f_{\R}(x)\in \re M'$ and  $f_i(x)\in \re{M'}\text{ for any } i=1,\dots,7$.

We now  give a characterization of    $\spo$-\almost linear maps in terms of the real part operators. The proof runs in completely the same manner as in our previous work \cite
{huoqinghai2021Riesz}.
\begin{thm}\label{thm: f in Hom(M,M')  equivalent:}
	Let $M$ be a left $\O$-module and $M'$ an $\O$-bimodule.  Then the following are equivalent:
	\begin{enumerate}
		\item  $f\in \Hom_\mathcal{LO}(M,M')$;
		\item $f_i(x)=f_{\R}(\overline{e_i}x)$ for any $i=0,\dots ,7$;
		\item $A_p(x,f)=\sum_{i=1}^7
		e_i f_{\R}([e_i,p,x])$.
	\end{enumerate}
	
\end{thm}
\begin{proof}
	We first prove that   a\asertion{1} implies a\asertion{2}.
	By definition, we have
	\begin{align*}
	f_\R(\overline{e_i}x)&=\re f(\overline{e_i}x)\\
	&=\re (\overline{e_i}f(x))\\
	&=\sum_{j=0}^7 \re \left(\overline{e_i} e_jf_j(x)\right).
	\end{align*}
	According to \eqref{eq:real part on good bimod }, we have
	\begin{align*}
	f_\R(\overline{e_i}x)&=\sum_{j=0}^7 \re \left(\overline{e_i}e_j\right)f_j(x)\\
	&=f_i(x).
	\end{align*}			
	This proves that a\asertion{1} implies a\asertion{2}.
 	We next come to prove that a\asertion{2} implies a\asertion{3}.
	\bfs\ $p=e_i$.
	From above computation and a\asertion{2}, we get
	\begin{align*}
	A_{e_i}(x,f)&=f(e_ix)-e_if(x)\\
	&=f_{\R}(e_ix)+\sum_{j=1}^7 e_jf_j(e_ix)-\left(-f_i(x)+e_if_{\R}(x)+\sum_{j,k=1}^7 \epsilon_{ijk}e_kf_j(x)\right)\\
	&=f_{\R}(e_ix)+f_i(x)-\sum_{j=1}^7 e_jf_{\R}(e_j(e_ix))-e_if_{\R}(x)-\sum_{j,k=1}^7 \epsilon_{ijk}e_kf_j(x)\\
	&=\left(-\sum_{j=1}^7 e_jf_{\R}\left( (e_je_i)x-[e_j,e_i,x] \right)\right)-e_if_{\R}(x)+\sum_{j,k=1}^7 \epsilon_{ijk}e_kf_{\R}(e_jx)\\
	&=\Big(-\sum_{j,k=1}^7 e_jf_{\R}\left( (\epsilon_{jik}e_k-\delta_{ij})x \right)\Big)+\sum_{j=1}^7 e_jf_{\R} ([e_j,e_i,x] )-e_if_{\R}(x)+\sum_{j,k=1}^7 \epsilon_{ijk}e_kf_{\R}(e_jx)\\
	&=\sum_{j=1}^7 e_jf_{\R} ([e_j,e_i,x] ).
	\end{align*} 			We   used the fact that $\epsilon_{ijk}$ is alternative in the last line.

	Finally we prove that  a\asertion{3} implies a\asertion{1}. This  is trivial in view of  Theorem \ref{lem:real part on good bimod }.  
\end{proof}

As an  important consequence of
Theorem \ref{thm: f in Hom(M,M')  equivalent:}, we find that  the second associator  vanishes for associative elements and  an $\spo$-\almost linear map $f$ is  uniquely determined by $f_\R$.

\begin{cor}\label{cor: A_p(x,f)=0 and f(px)=pf(x)}
	Assume that  $f\in \Hom_\mathcal{LO}(M,M')$.
	\begin{enumerate}
		\item For each \asselm\ $x\in \huaa{M}$, we have $$A_p(x,f)=0$$ and hence $f(px)=pf(x)$ for all $p\in \O$.
		\item  $f_{\R}=0\iff f= 0$.
		
	\end{enumerate}
	
\end{cor}
 \begin{cor}\label{lem: f(huaa M)=0 yields f=0}
	Let $M$  be an $\spo$-bimodule and $f\in  \Hom_{\mathcal{LO}}(M,M') $.
	\begin{enumerate}
		\item  If $f\lvert_{\re{M}}=0$, then $ f=0$;
		\item  $f\in  \Hom_{\spo}(M,M')$ if and only  if $ f\lvert_{\re{M}}\in \Hom_{\spr}(\re{M},\re{M'})$.
		\end{enumerate}
\end{cor}
\begin{proof}
	We first prove a\asertion{1}. Since $M$ is an $\O$-bimodule,  each $x\in M$   can be  expressed as $$x=x_0+\sum_{i=1}^7 e_ix_i,$$
	where $x_j\in\re M$  for any $j=0,\dots,7.$
	According to  a\asertion{1} of Corollary \ref{cor: A_p(x,f)=0 and f(px)=pf(x)}, we obtain $$f(x)=f(x_0)+\sum_{i=1}^7 e_i f(x_i),$$ so that the condition  $f\lvert_{\re{M}}=0$ will yield $f=0$.
	
	We come to prove a\asertion{2}.
	If $f\in  \Hom_{\spo}(M,M')$,  it follows from   Lemma \ref{lem:f re=re f} that $f(\re M)\subseteq\re M'.$
	
	Conversely,  if $f(\re M)\subseteq\re M'$, then  by a\asertion{1} of Corollary \ref{cor: A_p(x,f)=0 and f(px)=pf(x)} again, for all $p\in \spo$ we have	
	\begin{align*}
	f(px)&\xlongequal{}f(px_0)+\sum_{i=1}^7 f((pe_i)x_i)\\
	&\xlongequal{}pf(x_0)+\sum_{i=1}^7 (pe_i)f(x_i)\\
	&\xlongequal{}pf(x_0)+\sum_{i=1}^7 p(e_if(x_i))\\
	&\xlongequal{}pf(x).
	\end{align*}
	
\end{proof}

We now introduce two very useful bijections.
 \begin{lemma}
		If $M$ is a left $\O$-module and $M'$
		an $\O$-bimodule.
		Then we have a bijection of (left) \textbf{\almost linear lift} map:
		\begin{eqnarray}\label{eq:lif}
		l\text{-}\lif:\qquad\Hom_{\R}(M,\re M')&\to&\Hom_\mathcal{LO}(M,M')\notag\\
		g&\mapsto& ( l\text{-}\lif g)(x):=g(x)-\sum_{i=1}^7e_ig(e_ix).
		\end{eqnarray}
Its inverse map is given by
\begin{eqnarray}\label{def:lif-301}
		\re_*:\qquad\Hom_\mathcal{LO}(M,M')&\to&\Hom_{\R}(M,\re M')\notag\\
		f&\mapsto& \re \circ f=f_\R.
		\end{eqnarray}

\end{lemma}
\begin{proof}
		We first show that the $l\text{-}\lif$ map is well-defined.
		In view of Theorem \ref{thm: f in Hom(M,M')  equivalent:}, $l\text{-}\lif(g)$ is a left \almost linear map so that $l\text{-}\lif(g)\in \Hom_\mathcal{LO}(M,M')$.
		It is direct to check that $$\re_*\circ l\text{-}\lif=\text{id},\quad l\text{-}\lif\circ\re_* =\text{id}.$$	
\end{proof}

\begin{lemma}\label{lem:ext}	If $M$ and  $M'$ are two  $\O$-bimodules.
	Then we  have a bijection of (left) \textbf{\almost linear extension}:
	\begin{eqnarray}\label{eq:ext-302}
		l\text{-}\ext:\qquad\Hom_{\R}(\re M, M')&\to&\Hom_\mathcal{LO}(M,M')\notag \\
		g&\mapsto& l\text{-}\ext g,
	\end{eqnarray}
	which  the $l\text{-}\ext$ map   defined by
	\begin{eqnarray}\label{eq:ext-303}(l\text{-}\ext g)(px):=p(g(x))
\end{eqnarray}
	for all $p\in\O$ and all  $x\in \re M.$	
Its inverse map is given by 		\begin{eqnarray}
	\re^*:\qquad\Hom_\mathcal{LO}(M,M')&\to&\Hom_{\R}(\re M, M')\notag\\
	f&\mapsto& f\circ\re =f|_{\re M}.
	\end{eqnarray}
\end{lemma}
\begin{proof} Based on  Theorem \ref{lem:real part on good bimod },
the \almost linear extension is uniquely defined via
(\ref{eq:ext-303}).

We need to  prove that $l\text{-}\ext g\in \Hom_\mathcal{LO}(M,M')$ for any $g\in \Hom_{\spr}(\re M,M')$. Let $p\in \O$  and
	$$x=x_0+\sum_{i=1}^7e_ix_i\in M.$$
	Then by definition, we have
	\begin{align*}
	A_p(x,l\text{-}\ext g)&=(l\text{-}\ext g)(px)-p(l\text{-}\ext g)(x)\\
	&=(l\text{-}\ext g)(px_0+\sum_{i=1}^7(pe_i)x_i)-p(l\text{-}\ext g)(x_0+\sum_{i=1}^7e_ix_i)\\
	&=pg(x_0)+\sum_{i=1}^7(pe_i)g(x_i)-pg(x_0)-\sum_{i=1}^7p(e_ig(x_i))\\
	&=\sum_{i=1}^7[p,e_i,g(x_i)].
	\end{align*}
	It follows that $\re A_p(x,l\text{-}\ext g)=0$ and thus $l\text{-}\ext g\in \Hom_\mathcal{LO}(M,M')$.
	
By definition, one can check that $$\re^*\circ l\text{-}\ext=\text{id},\quad l\text{-}\ext\circ\re^* =\text{id}.$$	
This completes the proof. \end{proof}

	We shall denote $$\lif = l\text{-}\lif, \quad \ext=l\text{-}\ext$$ for simple if no confusion arises.

\subsection{Second associator of \almost linear maps}
In subsection \ref{sec:second ass}, we have studied the second associator of
real linear maps over $\O$-modules.
In this subsection, we focus on 	the properties of the second associator related to  $\spo$-\almost linear maps.
Some properties have appeared in \cite{huoqinghai2021Riesz} for $\spo$-\almost linear functions.

\begin{prop}\label{prop: left,bimod associator }
	Let $M$ be a  left $\spo$-module and $M'$ an $\spo$-bimodule. If  $f\in \Hom_\mathcal{LO}(M,M')$, then for any   $p,q\in \spo$ and $x\in M$ we have
	\begin{align}
	pA_p(x,f)&=A_p(x,f)\overline{p},\label{eq:Omod left bi pAp(xf)=Ap(xf)p}\\
	\re\left(A_p(x,f)q\right)&=-\re\left(A_q(x,f)p\right).\label{eq:Omod left bi Re(Ap(xf)q=-Re(Aq(xf)p))}
	\end{align}	
	
\end{prop}
\begin{proof}
	As usual,	we write  $$f(x)=f_{\R}(x)+\sum_{i=1}^7 e_if_i(x).$$

	We first prove
	\eqref{eq:Omod left bi pAp(xf)=Ap(xf)p}. \bfs\ $p\in \pureim{\spo}$. Let $p=\sum_{i=1}^7 p_ie_i$ with $ p_i\in \spr$. Then by Theorem \ref{thm: f in Hom(M,M')  equivalent:}, we have
	$$pA_p(x,f)=\sum_{i,j=1}^7 p_ie_iA_{p_je_j}(x,f)=\sum_{i,j,k=1}^7 p_ip_j e_i e_kf_{\R}([e_k,e_j,x]),$$
	and
	$$A_p(x,f)\overline{p}=\sum_{i,j=1}^7 A_{p_je_j}(x,f)\overline{(p_ie_i)}=-\sum_{i,j,k=1}^7 p_ip_je_ke_if_{\R}([e_k,e_j,x]).$$
	Since $$e_i e_k+e_ke_i=-2\delta_{ik},$$
	we conclude
	$$pA_p(x,f)-A_p(x,f)\overline{p}=-2\sum _{i,j=1}^7 p_ip_jf_{\R}([e_i,e_j,x])=0.$$
	The last step used the skew symmetricity of the associator $[e_i,e_j,x]$ in  $i,j$.

	To prove \eqref{eq:Omod left bi Re(Ap(xf)q=-Re(Aq(xf)p))}, we first  prove the  following identity
	\begin{eqnarray}
	\re\left(f([p,q,x])\right)=\re \left(A_p(x,f){q}\right)\label{eq:Omod left bi Re(f[pqx])=Re(Ap(x,f)q)}.
	\end{eqnarray}
	It suffices to show the case where  $q=e_i$ for each  $i=1,\dots,7$.
	By direct calculations, we have
	\begin{align*}
	\re \left(A_p(x,f){q}\right)&=\re\left(\sum_{j=1}^7 (e_jf_{\R}([e_j,p,x]))e_i\right)&\text{by Theorem \ref{thm: f in Hom(M,M')  equivalent:}}\\
	&=\re\left(\sum_{j=1}^7 f_{\R}([e_j,p,x]) e_je_i\right)&\text{by  \eqref{prop:re part}}\\
	&=-f_{\R}([e_i,p,x])&\text{by  \eqref{eq:real part on good bimod }}\\
	&=\re\left(f([p,q,x])\right).
	\end{align*}
	Due to \eqref{eq:Omod left bi Re(f[pqx])=Re(Ap(x,f)q)}, we conclude that  \eqref{eq:Omod left bi Re(Ap(xf)q=-Re(Aq(xf)p))} holds.
	
\end{proof}

\begin{prop}\label{prop:bimod,left,associator}
	Let  $M$, $M'$ be  two $\spo$-bimodules and $f\in \Hom_\mathcal{LO}(M,M')$.
	Then for all $p$, $q\in \spo$, and $ x\in M$ we have
	\begin{align}
	\re\left(f([p,x,q])\right)&=\re \left(pA_q(x,f)\right)\label{eq:Omod bi,left, Re f([p,x,q])=Re pA_q(x,f)}.
	\end{align}
\end{prop}
\begin{proof} As usual, we write
	$$f(x)=f_{\R}(x)+\sum_{i=1}^7 e_if_i(x),\qquad  x=x_0+\sum_{i=1}^7 e_ix_i.$$ 		
	\bfs\ that $p=e_i$ with $i\in \{1,\dots,7\}$.
		It follows from Theorem \ref{thm: f in Hom(M,M')  equivalent:} that
	\begin{align*}
	\re \left(e_iA_q(x,f)\right)&=\re \left(\sum_{j=1}^7 e_ie_jf_{\R}([e_j,q,x])\right)\\
	&=-f_{\R}([e_i,q,x])\\
	&=\re\left(f([e_i,x,q])\right).
	\end{align*}
	This proves \eqref{eq:Omod bi,left, Re f([p,x,q])=Re pA_q(x,f)}.
	
\end{proof}

\subsection{$\spo$-module structure of $\Hom_\mathcal{LO}(M,M')$ }

Now we can   provide $\Hom_\mathcal{LO}(M,M')$ with  an $\spo$-module structure.

\begin{thm}\label{Thm:left, bimod Hom(M,M')}
	Let $M$ be a left $\spo$-module and $M'$  an $\spo$-bimodule. Then $\Hom_\mathcal{LO}(M,M')$ admits a   right $\spo$-module structure with respect to  the right scalar multiplication  defined by \eqref{eqdef:<x,fr>}, i.e.,
	$$\left<x,f\odot r\right>:=\left<x,f\right>r-A_r(x,f)	$$
 	{for all $f\in \Hom_\mathcal{LO}(M,M')$, $r\in \spo$, and $ x\in M$.}
\end{thm}
\begin{proof}
	First we claim that $f\odot r\in \Hom_\mathcal{LO}(M,M') $ for all $r\in \O$. Recall  identity \eqref{eq:algmod A_p(x,fr)=[p,f(x),r]+A_p(x,f)r-A_r(px,f)+pA_r(x,f)}
	\begin{align*}
	A_p(x,f\odot r)&=[p,f(x),r]+A_p(x,f)r-A_r(px,f)+pA_r(x,f).
 	\end{align*}
	In order	to deduce $\re A_p(x,f\odot r)=0$,
	we only need to show that
	$$\re(A_p(x,f)r)=-\re(pA_r(x,f)).$$
	This follows from  identity \eqref{eq:Omod left bi Re(Ap(xf)q=-Re(Aq(xf)p))} and the fact $$\re (pA_r(x,f))=\re(A_r(x,f)p).$$
	
	Next we show    $\Hom_\mathcal{LO}(M,M')$ is  a right $\spo$-module. It suffices to prove $$(f\odot r)\odot r=f\odot r^2$$
	for any $r\in \O$. By definition \eqref{eqdef:<x,fr>}, we have
	\begin{align}
	\fx{x}{(f\odot r)\odot r}&=\fx{x}{f\odot r}r-A_r(x,f\odot r)\notag\\
	&=\fx{x}{f}r^2-A_r(x,f)r-A_r(x,f\odot r)\label{eq:<x,fr,r>}
	\intertext{and }
	\fx{x}{f\odot r^2}		&=	\fx{x}{f}r^2-A_{r^2}(x,f).\label{eq:<xfr2>}
	\end{align}
	Since $f,f\odot r\in \Hom_\mathcal{LO}(M,M')$, by Definition \ref{def:\almost linear} we have
	$$\re A_r(x,f\odot r)=0,\quad \re A_{r^2}(x,f)=0.$$
	Combining \eqref{eq:<x,fr,r>} and \eqref{eq:<xfr2>}, we get
	\begin{align*}
	\re \fx{x}{(f\odot r)\odot r-f\odot r^2}&=-\re(A_r(x,f)r)\\
	&=\re(A_{\overline{r}}(x,f)r)&\text{a\asertion{1} of Proposition  \ref{prop:left,left associator prop}}\\
	&=\re (\overline{r} A_{\overline{r}}(x,f))&\text{	by identity  \eqref{eq:Omod left bi pAp(xf)=Ap(xf)p}}\\
	&=\re A_{\overline{r}}(rx,f)&\text{a\asertion{3} of Proposition  \ref{prop:left,left associator prop}}\\
	&=0.
	\end{align*}
	Note that $$(f\odot r)\odot r-f\odot r^2\in \Hom_\mathcal{LO}(M,M').$$ It follows from  a\asertion{2} of  Corollary \ref{cor: A_p(x,f)=0 and f(px)=pf(x)} that $(f\odot r)\odot r=f\odot r^2$.	
	This completes the proof.
\end{proof}

 \begin{thm}\label{thm:O mod stct bi-lrft }
	Let  $M$, $M'$ be two $\spo$-bimodules. Then $\Hom_\mathcal{LO}(M,M')$ admits a left $\spo$-module structure with the left scalar multiplication  defined by  \eqref{eqdef:<x,rf>}, i.e.,
	$$	\left<x,r\odot f\right>:=\left<xr,f\right>+A_r(x,f)$$
	{for all $f\in \Hom_\mathcal{LO}(M,M')$, $r\in \spo$, and $ x\in M$.}
 \end{thm}

\begin{proof}
	First we claim that $$r\odot f\in \Hom_\mathcal{LO}(M,M') $$ for all $r\in \O$ and $f\in \Hom_\mathcal{LO}(M,M')$. Indeed, using identities  \eqref{eq:Omod bi,left, Re f([p,x,q])=Re pA_q(x,f)} and   \eqref{eq:Ap(x,rf)bileft}, we  obtain $$	\re \left(A_p(x,r\odot f)\right)=\re(f([p,x,r]))-\re(pA_r(x,f))=0.$$
 	
	Next we show that  $ \Hom_\mathcal{LO}(M,M')$ is  a left $\spo$-module. It suffices to prove $r\odot (r\odot f)=r^2\odot f$ for all $r\in \O$.
	For any $x\in M$, by definition \eqref{eqdef:<x,rf>} we have	
	\begin{align}\label{eq:r(rf)(x)}
	\fx{x}{r\odot(r\odot f)}=\fx{xr}{r\odot f}+A_r(x,r\odot f)
	=\fx{xr^2}{f}+A_r(xr,f)+A_r(x,r\odot f).
	\end{align}
	Since $f,r\odot f\in \Hom_\mathcal{LO}(M,M')$, we have
	\begin{align*}
	\re \fx{x}{r\odot(r\odot f)-r^2\odot f}&=\re\left(\fx{xr^2}{f}+A_r(xr,f)+A_r(x,r\odot f)\right)-\re\left(\fx{xr^2}{f}+A_{r^2}(x,f)\right)\\
	&=\re(A_r(xr,f)+A_r(x,r\odot f)-A_{r^2}(x,f))\\
	&=0.
	\end{align*}
	It  follows from   Corollary \ref{cor: A_p(x,f)=0 and f(px)=pf(x)} that $r\odot (r\odot f)=r^2\odot f$.	
	This finishes the proof.		
\end{proof}

As a consequence, we have  the following property about the second associators.

\begin{prop}\label{thm:bimod,left,associator}
	Let  $M$, $M'$ be two $\spo$-bimodules and $f\in  \Hom_\mathcal{LO}(M,M')$. Then for all , $ x\in M$ and $r\in \spo$ we have
	\begin{align}
	A_r(xr,f)&=rA_r(x,f);\label{eq:Omod bi,left, Ap(xp,f)=pAp(x,f)}  \\
	A_r(x,f)\overline{r}&=A_r(x{r},f).
	\end{align}
 \end{prop}

\begin{proof}
 	Let $f\in \Hom_\mathcal{LO}(M,M')$, $x\in M$ and $r\in \O$.
	Combining \eqref{eq:r(rf)(x)} with    \eqref{eq:Ap(x,rf)bileft}, we obtain 		$$\fx{x}{r\odot
		(r\odot f)}=\fx{xr^2}{f}+2A_r(xr,f)+A_r(rx,f)-rA_r(x,f).$$
	In view of a\asertion{2} of Proposition \ref{prop:left,left associator prop}, we get
	\begin{align*}
	\fx{x}{r^2\odot f}	
	&=\fx{xr^2}{f}+A_{r^2}(x,f)\\
	&=\fx{xr^2}{f}+rA_r(x,f)+A_r(rx,f).
	\end{align*}
	From Theorem \ref{thm:O mod stct bi-lrft }, we know that $r^2\odot f=r\odot (r\odot f)$.
	Therefore we have
	$$2A_r(xr,f)-rA_r(x,f)=rA_r(x,f).$$
	This implies  identity \eqref{eq:Omod bi,left, Ap(xp,f)=pAp(x,f)} holds.
	Combining  \eqref{eq:Omod left bi pAp(xf)=Ap(xf)p} with  \eqref{eq:Omod bi,left, Ap(xp,f)=pAp(x,f)}, we conclude
	$$A_p(x,f)\overline{p}={p}A_p(x,f)=A_p(x{p},f)$$ as desired.	
\end{proof}
 \begin{rem}
	Propositions	\ref{prop:left,left associator prop},  \ref{prop: left,bimod associator },      \ref{prop:bimod,left,associator} and  \ref{thm:bimod,left,associator} generalize Lemma 3.9 in  \cite{Grigorian2017octonionbundles} from the specific $\O$-module $\O$ itself  to the general $\O$-modules.
	
\end{rem}

\bigskip

By Theorems \ref{Thm:left, bimod Hom(M,M')} and   \ref{thm:O mod stct bi-lrft }, $ \Hom_\mathcal{LO}(M,M')$ admits a right $\spo$-module structure as well as a left $\spo$-module structure. Then a natural question arises whether the left and right scalar multiplications in $ \Hom_\mathcal{LO}(M,M')$ are compatible so that $ \Hom_\mathcal{LO}(M,M')$ admits an $\spo$-bimodule structure.
We shall  give an affirmative answer.

\begin{lemma}\label{lem:[p,q,f](x)=[p,q,f(x)],[q,f,p](x)=[q,f(x),p],[f,p,q](x)=[f(x),p,q]}
	Let  $M,$  $M'$ be two  $\spo$-bimodules and $f\in \Hom_\mathcal{LO}(M,M')$. Then with respect to the scalar multiplications  \eqref{eqdef:<x,fr>} and \eqref{eqdef:<x,rf>}, we have
	$$[p,q,f]=[q,f,p]=[f,p,q]$$
	for all $p,q\in \O$.
\end{lemma}
\begin{proof}
 	In view of   Corollary \ref{lem: f(huaa M)=0 yields f=0},  we only need to show that  $$[p,q,f](x)=[q,f,p](x)=[f,p,q](x) $$
	for all $x\in \re{M}$.
	
	Let $f\in \Hom_\mathcal{LO}(M,M')$ and fix an arbitrary $x\in \re M$.
	In view of   a\asertion{1} of Corollary \ref{cor: A_p(x,f)=0 and f(px)=pf(x)}, we immediately conclude from identity \eqref{eq:alm left bi[f,p,q](x)=[f(x),p,q]-Ap(x,f)q-Aq(x,fp)+A pq(x,f)} that
	$$[f,p,q](x)= [f(x),p,q].$$ 
	Similarly,  from \eqref{eq:alm bi left[pqf]x in } we have    $$[p,q,f](x)=-A_q(xp,f).$$
	Since   $x\in\re M$, it follows from Theorem \ref{lem:real part on good bimod } that
	\begin{align*}
	A_p(xq,f)&=\fx{p(xq)}{f}-p\fx{xq}{f}\\
	&=\fx{(pq)x}{f}-p\fx{qx}{f}\\
	&=(pq)f(x)-p(qf(x))\\
	&=[p,q,f(x)].
	\end{align*}
	This implies that  $${[p,q,f]}(x)=[p,q,f(x)]= [f(x),p,q]=[f,p,q](x).$$
	It remains  to show  $$[q,f,p](x)=[p,q,f](x).$$
	We conclude from    identity \eqref{eq:bi,bi [qfp]=Aq(xf)p-A_p(x,qf)+A_p(xq,f)-A_q(x,fp)}
 	that $${[q,f,p]}(x)=A_p(xq,f).$$
	Thus we obtain $[q,f,p](x)=[p,q,f](x)$ as desired.
\end{proof}
 As a consequence, we have

\begin{thm}\label{thm:Hom bimod}
	Let  $M,$  $M'$ be two $\spo$-bimodules. Then $\Hom_\mathcal{LO}(M,M')$ admits a  canonical  $\spo$-bimodule structure.
\end{thm}
 \subsection{Real part   of $\spo$-\almost linear maps}

We assume that   $M$, $M'$ are two $\spo$-bimodules in this subsection.  We come to study the real part structure of the $\O$-bimodule $ \Hom_\mathcal{LO} (M,M')$.

\begin{thm}\label{thm:Hom kelie left,bi}
	Let $M$, $M'$ be two $\spo$-bimodules. Then $$\re(\Hom_\mathcal{LO}(M,M'))=\Hom_{\spo}(M,M').$$
	More precisely, for all $f\in \Hom_\mathcal{LO}(M,M')$, we have
	 	\begin{eqnarray}\label{eq:re f=ext( fR| reM)}
	\re f=\ext (f_{\R}|_{\re M}).
	\end{eqnarray}
\end{thm}

\begin{proof}
	We first prove that
	\begin{eqnarray}\label{eq:left subset}
	\Hom_{\spo}(M,M')\subseteq\re{\Hom_\mathcal{LO}(M,M')}.
	\end{eqnarray}
	For any $f\in \Hom_{\spo}(M,M') $,  by definition we have $A_p(x,f)=0$ for all $p\in \spo$ and $x\in M$. Hence we conclude from identity \eqref{eq:alm left bi[f,p,q](x)=[f(x),p,q]-Ap(x,f)q-Aq(x,fp)+A pq(x,f)} that
	\begin{align*}
	[f,p,q](x)	&=[f(x),p,q]-A_q(x,f\odot p).
	\end{align*}
	This yields $$\re([f,p,q](x))=0$$ for all $x\in M$. Notice that
	$[f,p,q]\in \Hom_\mathcal{LO}(M,M')$ is a \almost linear map
	due to  Theorem \ref{Thm:left, bimod Hom(M,M')}. It follows from Corollary \ref{cor: A_p(x,f)=0 and f(px)=pf(x)} that 	$[f,p,q]=0$. This means
	$$f\in \huaa{\Hom_\mathcal{LO}(M,M')}.$$
	Since $\Hom_\mathcal{LO}(M,M')$ is an $\O$-bimodule, it follows from \eqref{eq:reM=huaaM=ZM} that
	$$f\in\re{\Hom_\mathcal{LO}(M,M')}.$$

	Conversely, if $f\in \re{\Hom_\mathcal{LO}(M,M')}$,
	then for all $x\in M$, $$[f,p,q](x)=0.$$
	 	Fix $x\in \re{M}$.  It follows from identity \eqref{eq:alm left bi[f,p,q](x)=[f(x),p,q]-Ap(x,f)q-Aq(x,fp)+A pq(x,f)} and Corollary \ref{cor: A_p(x,f)=0 and f(px)=pf(x)} that
	$$[f(x),p,q]=A_p(x,f)q+A_q(x,f\odot p)-A_{pq}(x,f)=0$$ for all $p,q\in\spo$.
	This means $f(x)\in \re{M'}$ and hence  we obtain
	$$f(\re{M})\subseteq \re{M'}.$$ By a\asertion{2} of Corollary \ref{lem: f(huaa M)=0 yields f=0}, we conclude that  $f\in \Hom_{\spo}(M,M')$. This proves
	\begin{eqnarray}\label{eq:right subset}
	\re{\Hom_\mathcal{LO}(M,M')}\subseteq\Hom_{\spo}(M,M').
	\end{eqnarray}
	Combining \eqref{eq:left subset} and \eqref{eq:right subset}, we conclude
	$$\re{\Hom_\mathcal{LO}(M,M')}=\Hom_{\spo}(M,M').$$
	
	We now prove	\eqref{eq:re f=ext( fR| reM)}.
	For any $f\in \Hom_\mathcal{LO}(M,M')$, we get
	an $\R$-linear map $$f_{\R}|_{\re M} :\re M\to M'.$$ By the definition of the $\ext$ map,  we have $$\ext (f_{\R}|_{\re M})\in \Hom_\mathcal{LO}(M,M').$$
	Because of  Corollary \ref{lem: f(huaa M)=0 yields f=0}, to show \eqref{eq:re f=ext( fR| reM)} it suffices to show that $$\fx{x}{\re f}=\fx{x}{f_{\R}|_{\re M}}$$ for all $x\in \re M$.  Fix an arbitrary $x\in \re M$.	Since $\re f $ is $\O$-linear, we conclude  that $$(\re f)(x)\in \re M'.$$
	Therefore,
	\begin{align*}
	\fx{x}{\re f}&=\re\fx{x}{\re f}&\text{$\re$ is a projection operator}\\
	&=\re\Big(\langle{x}, {\frac{5}{12}f}\rangle-\langle{x}, {\frac{1}{12}\sum_{i=1}^7 e_i\odot f\odot e_i}\rangle\Big)&\text{by  \eqref{eq:def of real part}}\\
	&=\frac{5}{12}f_{\R}(x)-\frac{1}{12}\re \sum_{i=1}^7 f(xe_i)e_i&\text{using  definitions \eqref{eqdef:<x,fr>} and \eqref{eqdef:<x,rf>}}\\
	&=\frac{5}{12}f_{\R}(x)-\frac{1}{12}\re \sum_{i=1}^7 e_if(xe_i)&\text{using   \eqref{prop:re part}}\\
	&=\frac{5}{12}f_{\R}(x)-\frac{1}{12}\sum_{i=1}^7 f_{\R}(e_ixe_i)&\text{using $f\in \Hom_\mathcal{LO}(M,M')$}\\
	&=f_{\R}(\re x)\\
	&=f_{\R}(x)\\
	&=f_{\R}|_{\re M}(x).
	\end{align*}
	This completes the proof.
\end{proof}

\begin{rem}
	It	follows from a\asertion{3} of Theorem \ref{lem:real part on good bimod } that
 	for any $f\in \Hom_\mathcal{LO}(M,M')$, we have
	$$f=\sum_{i=0}^7 e_i\odot\re (\overline{e_i}\odot f)=:\sum_{i=0}^7 e_i\odot f_{(i)}.$$

	It is worth pointing out that the following seemingly natural equality
	$$(e_i\odot f_{(i)})(x)=e_if_{(i)}(x)$$   is not correct since the map under considered is left para-linear.  It  should be modified as
	\begin{eqnarray}\label{eq:eif(x)=f(x)ei}
	(e_i\odot
	f_{(i)})(x)=f_{(i)}(x)e_i.
	\end{eqnarray}
 	In fact, since  $f_{(i)}\in \Hom_{\spo}(M,M')$ is $\O$-linear, we have
	\begin{align*}
	(e_i\odot f_{(i)})(x)
	&=\fx{xe_i}{f_{(i)}}+A_{e_i}(x,f_{(i)})&\text{by definition \eqref{eqdef:<rf,x>}}\\
	&=f_{(i)}(xe_i)&\text{using $A_{e_i}(x,f_{(i)})=0$}\\
	&=f_{(i)}(x)e_i&\text{using  \eqref{eq:HomO(M,M')=l-Hom(M,M')}.}
	\end{align*}	
	
 \end{rem}

Theorem \ref{thm:Hom kelie left,bi} leads to an important and useful uniqueness result.
\begin{cor}\label{cor:f in im (U) = f(x) in im (M)}
	Under the  assumptions above, we have
	$$\re f=0
	\mbox{\ if \ and \ only \ if \ }
	\re f(x)=0\text{ for all }x\in \re {M}.$$
 \end{cor}
\begin{proof}
	The condition in the right hand side above is equivalent to
	\begin{eqnarray}\label{eq:proof ref =0}
	f_{\R}|_{\re M}=0.
	\end{eqnarray}
	It follows  from  Corollary \ref{lem: f(huaa M)=0 yields f=0} and \eqref{eq:re f=ext( fR| reM)}  that the condition $\re f=0$ is also equivalent to \eqref{eq:proof ref =0}. This completes the proof.
 \end{proof}

\subsection{Regular composition of $\spo$-\almost linear maps}
Since  the \almost linearity does not preserve under the ordinary   composition of two \almost linear maps,  we need to introduce a new   composition.

In this subsection, we always assume that 	 $M,\ M'$ are two left $\spo$-modules and $\ M''$ is an $\spo$-bimodule.

\begin{mydef}\label{def:regular composition}
	Let  $f\in\Hom_{\spr}(M',M'')$ and  $g\in  \Hom_{\spr}(M,M')$.
	The \textbf{(left) regular composition} of $f$ and $g$ is defined as
	$$f\circledcirc_l g(x):=f\circ g(x)+[f,g,x],$$ where
	\begin{eqnarray}\label{def:[fgx]}
	[f,g,x]:=-\sum_{j=1}^7 e_j\re \big(f(A_{e_j}(x,g))\big).
	\end{eqnarray}
	
\end{mydef}
\begin{rem}
	We shall just denote $f\circledcirc_l g=f\circledcirc g$ if there is no confusion.
	By direct calculation, we have an identity about associators concerning five terms:
	\begin{equation}\label{eq:Ap(x,f odot g)=Ap(g(x),f)+f(Ap(x,g))+[f,g,px]-p[f,g,x]}
	A_p(x,f\circledcirc g)=A_p(g(x),f)+f(A_p(x,g))+[f,g,px]-p[f,g,x].
	\end{equation}
	We shall show in		Proposition \ref{prop:f.g in Hom(M,M')} that the \almost linearity preserves under the regular compositions. This fact justifies our   Definition \ref{def:regular composition}. The point is that the regular composition is non-associative in contrast to the classical composition. This new composition plays a central role in the theory of  non-associative categories.
\end{rem}
We give a sufficient  condition for  the regular composition becoming  the ordinary  composition.
\begin{lemma}\label{lem:vanishing[f,g,x]=0}
	Let  $f\in\Hom_{\spr}(M',M'')$ and  $g\in  \Hom_{\spr}(M,M')$. Suppose   one of the following conditions holds
	\begin{enumerate}
		\item $g\in \Hom_{\spo}(M,M')$;
		\item $M'$ is an $\O$-bimodule, $x\in \re{M}$ and  $g\in \Hom_\mathcal{LO}(M,M')$;
		\item $M'$ is an $\O$-bimodule, $f\in  \Hom_{\spo}(M',M'')$ and  $g\in \Hom_\mathcal{LO}(M,M')$.

	\end{enumerate}
	We then  have $$[f,g,x]=0$$ for all $x\in M$ and	
	$$ f\circledcirc g=f\circ g.$$	
\end{lemma}
\begin{proof}
	It suffices to show that $[f,g,x]=0$ for all $x\in M$ in each case.
	
	If  $g\in \Hom_{\spo}(M,M')$ then $[f,g,x]=0$ due to the fact that $A_{e_j}(x,g)=0$. 
	
	Let $M'$ be an $\O$-bimodule. If $x\in \re{M}$ and  $g\in \Hom_\mathcal{LO}(M,M')$, by Corollary \ref{cor: A_p(x,f)=0 and f(px)=pf(x)} we conclude $$A_{e_j}(x,g)=0.$$ This yields   $[f,g,x]=0$.
	
	If $f\in  \Hom_{\spo}(M',M'')$ and  $g\in \Hom_\mathcal{LO}(M,M')$, we then conclude from Lemma \ref{lem:f re=re f} that
	$$\re f(A_{e_j}(x,g))=f(\re A_{e_j}(x,g))$$ and from   Definition \ref{def:\almost linear} that
	$$\re A_{e_j}(x,g)=0.$$
	Therefore,   $$\re f(A_{e_j}(x,g))=f(\re A_{e_j}(x,g))=0.$$ Thus by \eqref{def:[fgx]}, we have  $[f,g,x]=0$. This proves the lemma.
\end{proof}

\begin{prop}\label{prop:f.g in Hom(M,M')}
	If  $f\in  \Hom_\mathcal{LO}(M',M'')$ and  $g\in \Hom_{\spr}(M,M')$, then	$$f\circledcirc g\in \Hom_\mathcal{LO}(M,M'').$$
\end{prop}
\begin{proof}
	We need to show   that $$\re A_p(x,f\circledcirc g)=0$$ for all $p\in \O$ and $x\in M$.
	In view of identity \eqref{eq:Ap(x,f odot g)=Ap(g(x),f)+f(Ap(x,g))+[f,g,px]-p[f,g,x]}, it is sufficient  to prove that $$\re(f(A_p(x,g))-p[f,g,x])=0.$$
		\bfs \ $p=e_k$.
	In view of  \eqref{lem:real part on good bimod },
 	\begin{align*}
	\re\big(f(A_{e_k}(x,g))-e_k[f,g,x]\big)
	=\re f(A_{e_k}(x,g))+\re\Bigl( e_k\sum_{j=1}^7 e_j\re \big(f(A_{e_j}(x,g))\big)\Bigl)
	=0.	
	\end{align*}	
	This completes the proof.
\end{proof}

{ For  later use, we give some relations among various associators.
}

\begin{lemma}\label{lem:[fgh]=}
	{	Let $M,M',M'',M'''$ be  $\O$-bimodules and 		
		let  $$h\in \Hom_\mathcal{LO}(M,M'), \quad g\in \Hom_\mathcal{LO}(M',M''), \quad f\in \Hom_\mathcal{LO}(M'',M''').$$
		Then for all $x\in M$, we have
		\begin{align}\label{eq:[f,g,h]=}
		[f,g,h](x)+f[g,h,x]=[f,g,h(x)]-[f,g\circledcirc h,x]+[f\circledcirc g,h,x ],
		\end{align}
		where $$[f,g,h]:=(f\circledcirc g)\circledcirc h-f\circledcirc (g\circledcirc h).$$}
	
	{	Moreover, we have the following properties.
		\begin{itemize}
			\item Let $f,g,h$ be as above. Then 	$$\re[f,g,h]=0;$$
			\item 	If   one of $f, g, h$  is $\O$-linear, then  $$[f,g,h]=0.$$
		\end{itemize}
	}
	
\end{lemma}
 \begin{proof}
	Identity  \eqref{eq:[f,g,h]=} follows by direct calculation.
	Combining \eqref{eq:[f,g,h]=} with a\asertion{2} of  Lemma \ref{lem:vanishing[f,g,x]=0}, we  deduce that
	$$\re ([f,g,h](x))=\re [f,g,h(x)]=0 \text{ for all }x\in \re{M}.$$
	It  then follows from Corollary \ref{cor:f in im (U) = f(x) in im (M)} that   $\re[f,g,h]=0$.

	Suppose  $f$ is an $\O$-linear map.
		For any $x\in M$, by identity  \eqref{eq:[f,g,h]=} and a\asertion{3} of Lemma \ref{lem:vanishing[f,g,x]=0}, we have
	\begin{eqnarray}\label{pfeq:[fgh](x)}
	[f,g,h](x)=[f\circ g,h,x]-f([g,h,x]).
	\end{eqnarray}
		Then by definition \eqref{def:[fgx]}, the right hand side above  becomes
	$$\sum_{i=1}^7 e_i\re f\circ g(A_{e_i}(x,h))-f\left(\sum_{i=1}^7 e_i\re  g(A_{e_i}(x,h))\right).$$
	In view of a\asertion{1} of Corollary \ref{cor: A_p(x,f)=0 and f(px)=pf(x)}, the second term above becomes
	\begin{eqnarray}\label{eqpf:fsumgx}
	f\left(\sum_{i=1}^7 e_i\re  g(A_{e_i}(x,h))\right)=\sum_{i=1}^7 e_if(\re  g(A_{e_i}(x,h))).
	\end{eqnarray}
	By Lemma \ref{lem:f re=re f}, we conclude
	\begin{eqnarray}\label{eqpf:sumfg}
	\sum_{i=1}^7 e_i\re f\circ g(A_{e_i}(x,h))=\sum_{i=1}^7 e_i f(\re g(A_{e_i}(x,h))).
	\end{eqnarray}
	Combining \eqref{pfeq:[fgh](x)}, \eqref{eqpf:fsumgx}
	and \eqref{eqpf:sumfg}, we obtain	 $$[f,g,h](x)=0.$$
	
	We have proved the  identity above when $f$ is $\O$-linear.
	 \end{proof}

Finally, we establish some properties of the right multiplication operators in an $\O$-bimodule $M$
\begin{eqnarray*}
R_p:M&\to& M\\
x&\mapsto &xp
\end{eqnarray*}
 for the use in the sequel.
\begin{lemma}\label{lem:Rp prop1}
	Let  $M$ be a left $\spo$-module and  $M'$  an $\spo$-bimodule. Then for all $p\in \spo$, $ x\in M$ and $ f\in \Hom_\mathcal{LO}(M,M')$, we have
	\begin{enumerate}
		\item $[R_p,f,x]=-A_p(x,f)$;
		\item  $R_p\circledcirc f=f\odot p$.

	\end{enumerate}
	
\end{lemma}
\begin{proof}
	We first prove a\asertion{1}.  From   \eqref{eq:Omod left bi Re(Ap(xf)q=-Re(Aq(xf)p))}, we have
	\begin{align*}
	[R_p,f,x]&=-\sum_{j=1}^7 e_j\re \big(A_{e_j}(x,f)p\big)\\
	&=\sum_{j=1}^7 e_j\re \big(A_{p}(x,f)e_j\big)\\
	&=-A_{p}(x,f).
	\end{align*}	
	Now we prove a\asertion{2}. By definition \eqref{eqdef:<x,fr>}, we get
	\begin{align*}
	R_p\circledcirc f(x)=f(x)p+[R_p,f,x]
	=f(x)p-A_{p}(x,f)
	=\fx{x}{f\odot p}.
	\end{align*}
	This completes the proof.
\end{proof}
\begin{lemma}\label{lem:Rp prop 2}
	Let  $M$, $M'$ be two  $\spo$-bimodules. Then for all $p\in \spo,\, x\in M$ and $ f\in \Hom_\mathcal{LO}(M,M')$, we have
	\begin{enumerate}
		\item $[f,R_p,x]=A_p(x,f)$;
		\item  $f\circledcirc R_p=p\odot
		f$.

	\end{enumerate}
\end{lemma}
\begin{proof}
	We first prove a\asertion{1}. Notice that  $$A_{e_j}(x,R_p)=[e_j,x,p]$$for each $j=1,\dots,7$.
	Combining this with \eqref{eq:Omod bi,left, Re f([p,x,q])=Re pA_q(x,f)}, we get
	\begin{align*}
	[f,R_p,x]&=-\sum_{j=1}^7 e_j\re \big( f(A_{e_j}(x,R_p))\big)\\
	&=-\sum_{j=1}^7 e_j\re \big(f([e_j,x,p])\big)\\
	&=-\sum_{j=1}^7 e_j\re \big(e_jA_{p}(x,f)\big)\\
	&=A_{p}(x,f).
	\end{align*}
	
	We come to prove a\asertion{2}. By definition \eqref{eqdef:<x,rf>}, we have
	\begin{align*}
	f\circledcirc R_p=f(xp)+[f,R_p,x]
	=f(xp)+A_{p}(x,f)
	=\fx{x}{p\odot
		f}.
	\end{align*}
	This finishes the proof.	
\end{proof}
 \subsection{Right \almost linear maps}\label{sec:right O mod}
The results of right \almost linear maps are all parallel with the left cases. Given a right $\O$-module, we can endow it with a natural left  $\O$-module structure so that we can interpret a    right \almost linear map as a left \almost linear map. Thus we can establish these results    via duality.

Let $M$ be a right $\O$-module.
It induces a left scalar multiplication on $M$ as follows:
$$p\cdot x:=x\overline{p},\quad\text{for all $x\in M$, $p\in \O$}.$$
One can check that this is indeed a left $\O$-module structure on $M$.
Similarly, a left $\O$-module also can induce a right $\O$-module structure. And  an $\O$-bimodule  can induce a new $\O$-bimodule structure.
 \textbf{The induced module is denoted by  $M^{C}$ for all cases.}

Let  $M$ and $N$ be two   right $\O$-modules. For any	
$$f\in \Hom_\R(M,N),$$   we can regard $f$ as a map between the left $\O$-modules $M^{C}$ and $N^{C}$, denoted by
$$f^{C}\in \Hom_\R(M^{C}, N^{C}).$$
Moreover, their associators     have the following relation
$$B_p(f,x)=A_p(x,f^{C})$$ for any $x\in M$ and $p\in \O$.	
This can be checked  by direct calculation
\begin{align*}
f^{C}(p\cdot x)
= f(x\overline{p})
= f(x)\overline{p}-B_{\overline{p}}(f,x)
=p\cdot f^{C}(x)+B_p(f,x).
\end{align*}

\bigskip

Now we can state the dual version of  Theorem \ref{thm: f in Hom(M,M')  equivalent:}.

\begin{thm}\label{thm: f in right-Hom(M,M')  equivalent:}
	Let $M$ be a right $\spo$-module and $M'$ be  an $\O$-bimodule.
	Suppose  that  $f\in \Hom_{\spr}(M,M')$ and write $$f(x)=f_{\R}(x)+\sum_{i=1}^7f_i(x)e_i.$$ Then the following are equivalent:
	\begin{enumerate}
		\item  $f\in  \Hom_\mathcal{RO}(M,M')$;
		\item $f_i(x)=-f_{\R}(xe_i),\quad i=1,\dots ,7$;
		\item $B_p(f,x)=\sum_{i=1}^7 f_{\R}([x,p,e_i])e_i$.
	\end{enumerate}
\end{thm}
{Similarly, we can also define     the related bijections of right  \almost linear lift and  right \almost linear extension}.
 For a real linear map $f\in \Hom_\R(M,N)$, which can be viewed as $f^{C}\in \Hom_\R(M^{C}, N^{C})$ as well, it follows from definitions \eqref{eqdef:<x,fr>} and \eqref{eqdef:<rf,x>} that
$$\fx{x}{f^{C}\odot
	r}=f^{C}(x)r-A_r(x,f^{C})=\overline{r}\cdot f(x)+B_{\overline{r}}(f,x)=\fx{\overline{r} \odot
	f}{x}.$$
Hence  it holds $$f^{C}\odot r=\overline{r}\odot  f.$$
 Similarly, we also have $$r\odot f^{C}=f \odot \overline{r}.$$ Therefore, we have the following dual versions of   Theorems \ref{Thm:left, bimod Hom(M,M')},  \ref{thm:Hom bimod} and   \ref{thm:Hom kelie left,bi}.

\begin{thm}\label{Thm:right, bimod Hom(M,M')}
	Let $M$ be a right $\spo$-module and  $M'$  an $\spo$-bimodule. Then $\Hom_\mathcal{RO}(M,M')$ admits a left $\spo$-module structure with the left scalar multiplication  defined by \eqref{eqdef:<rf,x>}.	
 	
\end{thm}

 \begin{thm}
	Let  $M,$  $M'$ be two  $\spo$-bimodules. Then $\Hom_\mathcal{RO}(M,M')$ admits an   $\spo$-bimodule structure  with the left scalar multiplication  defined by \eqref{eqdef:<rf,x>} and the right scalar multiplication  defined by  \eqref{eqdef:<fr,x>}.
\end{thm}
 \begin{thm}\label{thm:r-Hom(M,M') free}
	Let  $M$, $M'$ be two  $\spo$-bimodules. Then
	\begin{align}	\re{\Hom_\mathcal{RO}(M,M')}&=\Hom_{\spo}(M,M')\label{seteq:huaal(rHom(M,M)=HomO(MM))}
	\end{align}
	 \end{thm}

{We now define  right regular compositions also in a dual manner.}	
\begin{mydef}\label{def:right mod regular composition}
	Let $M,\ M'$ be two right $\spo$-modules and $\ M''$  an $\spo$-bimodule. If  $f\in  \Hom_{\spr}(M',M'')$ and $g\in  \Hom_{\spr}(M,M')$,
	we define their \textbf{(right) regular composition} as
	$$(f\circledcirc_r g)\,(x):=(f\circ g)(x)+[x,f,g],$$ where $$  [x,f,g]:=\sum_{j=1}^7 \Big( \re f\big(B_{e_j}(g,x)\big)\Big)e_j.$$
	
\end{mydef}
\begin{rem}\label{rem:[x,f,g]=[f,g,x]}
	
	By direct calculation,  for right regular compositions we have
	\begin{equation}\label{eq:Bp(f odot g,x)=B_p(f,g(x))+f(B_p(g,x))-[f,g,xp]+[f,g,x]p}
	B_p(f\circledcirc_r g,x)=B_p(f,g(x))+f(B_p(g,x))-[xp,f,g]+[x,f,g]p.
	\end{equation}
	We point out that the associator defined in the regular composition of right \almost linear maps are the same with the left one essentially.  Indeed,    it holds
	\begin{align*}
	[x,f,g]=[f^{C},g^{C},x].
	\end{align*}
	This can be checked   by direct calculations.
	
	We shall omit the subscripts $``l"$ and $``r"$ in regular compositions for simplicity, if the module is understood well.
\end{rem}

\

We now discuss the relations between left and right \almost linear maps over $\O$-bimodules.
 We first introduce a notion of \textbf{transpose} of \almost linear maps, which makes us  transform a left \almost linear map into a right one. Let  $M$, $M'$ be two  $\spo$-bimodules from now on.

\begin{mydef}\label{def:transpose}
	Let $f\in \Hom_\mathcal{LO}(M,M')$ be a left \almost linear map. A right \almost linear  map $g\in  \Hom_\mathcal{RO}(M,M')$ is called a \textbf{(right)  transpose} of $f$ if $$\re\big(f(x)-g(x)\big)=0  \text{ \quad for\ all } x\in M.$$
	One can define   \textbf{(left)  transpose} similarly.
\end{mydef}
\begin{rem}
	By Theorems \ref{thm: f in Hom(M,M')  equivalent:} and  \ref{thm: f in right-Hom(M,M')  equivalent:}, there exists a unique (right)  transpose for every left \almost linear map $f\in \Hom_\mathcal{LO}(M,M')$, denoted by $\hat{f}$. More precisely,
	\begin{eqnarray}\label{eq:hat f}
	\hat{f}=r\text{-}\lif(f_\R).
	\end{eqnarray}

	Similar results hold for left  transpose. The left  transpose of a right \almost linear map $f$ will still be denoted by $\hat{f}$.
\end{rem}

The transpose map turns out to be  an    isomorphism.
\begin{thm}
	Let $M,M'$ be two $\O$-bimodules. Then
	\begin{eqnarray*}\varphi:\Hom_\mathcal{LO}(M,M')&\to& \Hom_\mathcal{RO}(M,M')\\
		f &\mapsto&  \hat{f}
	\end{eqnarray*} is an  isomorphism of $\O$-bimodules. 
\end{thm}
\begin{proof} It is easy to see that
	$\hat{\hat{f}}=f$.  It remains to show  that  $\varphi$ is an $\O$-homomorphism.
	Let $f\in \Hom_\mathcal{LO}(M,M')$ be a left \almost linear map. For any $x\in \re M$, it follows from  Theorem \ref{thm: f in Hom(M,M')  equivalent:} that
	\begin{align*}
	\fx{\varphi(e_i\odot f)}{x}&=\re\fx{x}{e_i\odot f}-\sum_{j=1}^7 e_j\re\fx{xe_j}{e_i\odot f}\\
	&=f_{\R}(xe_i)-\sum_{j,k=1}^7 e_jf_{\R}(xe_k\epsilon_{jik}-\delta_{ij}x)\\
	&=f_{\R}(xe_i)+e_if_{\R}(x)-\sum_{j,k=1}^7\epsilon_{jik} e_jf_{\R}(xe_k),\intertext{and similarly, }
	\fx{e_i\odot\varphi (f)}{x}&=e_i\big(f_{\R}(x)-\sum_{j=1}^7 e_jf_{\R}(xe_j)\big)\\
	&=e_if_{\R}(x)-\sum_{j,k=1}^7 (\epsilon_{ijk}e_k-\delta_{ij})f_{\R}(xe_j)\\
	&=f_{\R}(xe_i)+e_if_{\R}(x)-\sum_{j,k=1}^7\epsilon_{ijk} e_kf_{\R}(xe_j).
	\end{align*}
	Symmetry in  $i,\, j$  ensures that $	\fx{\varphi(e_i\odot f)}{x}=	\fx{e_i\odot\varphi (f)}{x}$ for $x\in \re M$  and thus
	$$\varphi(e_i\odot f)=e_i\odot \varphi (f)$$
	for $i=1,\dots,7$. This shows that $\varphi$ is an $\O$-homomorphism  by \eqref{eq:HomO(M,M')=l-Hom(M,M')}.
\end{proof}

 The common subset of    $	\Hom_\mathcal{LO}(M,M')$ and $ \Hom_\mathcal{RO}(M,M')$   is exactly the set of  $\O$-linear maps.
\begin{thm}\label{Thm:relation-left-R223}
	Let  $M$, $M'$ be two  $\spo$-bimodules. Then  	
	\begin{align}
	\Hom_\mathcal{LO}(M,M')\cap \Hom_\mathcal{RO}(M,M')=\Hom_{\spo}(M,M')\label{seteq:rHom cap lHom=HomO}
	\end{align}
\end{thm}
\begin{proof}	
	According to identity \eqref{eq:HomO(M,M')=l-Hom(M,M')}, we  have
	\begin{eqnarray}\label{eq:cap of Hom subset HomO}
	\Hom_\mathcal{LO}(M,M')\cap \Hom_\mathcal{RO}(M,M')\supseteq \Hom_{\spo}(M,M').
	\end{eqnarray}
	
	Conversely, let $f\in \Hom_\mathcal{LO}(M,M')\cap \Hom_\mathcal{RO}(M,M')$ and choose $x\in \re{M}$.  Then for any $p\in \spo$, we have
	\begin{align*}
	[f(x),p]&=f(x)p-pf(x)\\
	&=f(xp)+B_p(f,x)-f(px)+A_p(x,f)\\
	&=B_p(f,x)+A_p(x,f)+f([x,p])\\
	&=B_p(f,x)+A_p(x,f).
	\end{align*}
	The last step follows from \eqref{eq:reM=huaaM=ZM}.
	By Corollary \ref{cor: A_p(x,f)=0 and f(px)=pf(x)}, we conclude from $x\in \re M$ that  $A_p(x,f)=0$, and similarly    $B_p(f,x)=0$. Hence we obtain $[f(x),p]=0$ for all $p\in \O$. This yields $f(x)\in \re M'$  by \eqref{eq:reM=huaaM=ZM}.
	Then Corollary \ref{lem: f(huaa M)=0 yields f=0} ensures that $f\in \Hom_{\spo}(M,M') $.
		This completes the proof.
	\end{proof}
\begin{rem}
		Recall definitions \eqref{eqdef:<x,fr>}, \eqref{eqdef:<x,rf>} \eqref{eqdef:<rf,x>} and  \eqref{eqdef:<fr,x>}.
	Starting from the relation \eqref{seteq:rHom cap lHom=HomO}, we observe that  an $\O$-linear function $f$ between $\O$-bimodules can be considered as either types of left or right \almost linear map, so that it  admits four different scalar multiplications. More precisely, if we regard $f\in \Hom_\mathcal{LO}(M,M')$, then for any $r\in \O$,
	\begin{eqnarray*}
		\left<x,f\odot r\right>=\left<x,f\right>r,\quad \left<x,r\odot f\right>=\left<xr,f\right>;
	\end{eqnarray*}
	and if we regard $f\in  \Hom_\mathcal{RO}(M,M')$, then for any $r\in \O$,
	\begin{eqnarray*}
		\left<r\odot f,x\right>=r\left<f,x\right>,\quad	\left<f\odot r,x\right>=\left<f,rx\right>.
	\end{eqnarray*}
	Note that all the related associators in definitions vanish since $f$ is $\O$-linear.

\end{rem}

\subsection{Semi-reflexivity}\label{sec:semi-reflex}

In this subsection, we shall show that octonionic modules are semi-reflexive. We denote by $M^{*_l}$ ($M^{*_r}$) the dual module $\Hom_{\mathcal{LO}}(M,\O)$ ($\Hom_{\mathcal{RO}}(M,\O)$) of an $\O$-module $M$ respectively. And we shall omit the subscript `l', `r' when the module is understood well.

Let $M$ be a left $\O$-module and
 $x_0$ be an arbitrary element of $M$. We define
\begin{eqnarray*}
	x_0'':\quad M^*&\to & \O\\
	f&\mapsto& f(x_0).
\end{eqnarray*}
It follows from Theorem \ref{Thm:left, bimod Hom(M,M')} that
 $M^*$ is a right $\spo$-module and similarly $M^{**}$ is a left $\spo$-module. We claim that $x_0''\in M^{**}$ is a   {right \almost linear} map on $M^*$. To show this, we first compute the right associator
\begin{align}\label{eq:Bp(f,x)=A(x,f)}
B_p(x_0'',f)=x_0''(f)p-x_0''(f\odot p)
=f(x_0)p-\fx{x_0}{f\odot p}
=A_p(x_0,f)
\end{align}
for any $f\in M^*$ and  any $p\in \spo$.
We  used definition \eqref{eqdef:<x,fr>} in the last equality.
This shows that $$\re (B_p(x_0'',f))=0$$ and hence $x_0''$ is a right $\O$-\almost linear map.

We define
\begin{eqnarray}\label{defeq:embedding tau}
\tau:\quad M&\rightarrow& M^{**} \\
x&\mapsto& x''.\notag
\end{eqnarray}
Then $\tau$ is an \textbf{embedding $\O$-homomorphism}. Indeed,
by definition, we have
\begin{align*}
\fx{\tau(rx)}{f}=\fx{rx}{f}=rf(x)+A_r(x,f)
\end{align*}
for any $r\in \O$, any $x\in M$ and any $f\in M^*$.
Next, by definition \eqref{eqdef:<rf,x>} and  identity \eqref{eq:Bp(f,x)=A(x,f)}, we obtain
\begin{align*}
\fx{r\odot\tau(x)}{f}=r\fx{\tau(x)}{f}+B_r(\tau(x),f)=rf(x)+A_r(x,f).
\end{align*}
Thus $\tau(rx)=r\odot \tau(x)$.
This shows that $\tau$ is an $\O$-homomorphism.	Thus each left $\O$-module $M$ is a semi-reflexive module.
\begin{rem}
	Let $\tau _\spr$ denote the canonical embedding $$\tau _\spr: M\rightarrow M^{*_\spr*_\spr}.$$ Then for all  $f\in M^*$ and $x\in M$ we have
	\begin{eqnarray}\label{eq:tau(x)0(f)=tauR(x)(f0)}
	\fx{\tau(x)_{\R}}{f}=\re (f(x))=f_{\R}(x)=\fx{\tau_\spr(x)}{f_{\R}}.
	\end{eqnarray}
\end{rem}

\section{Non-associative categories of \almost linear maps}\label{sec:cat}

In this section, we aim to adopt the view of  non-associative categories to study    \almost linear maps on $\O$-modules.
The category $\mathcal{LO}$, which will be introduced later, is a typical example of non-associative categories, whose objects in this category are $\O$-bimodules, morphisms are left \almost linear maps, and compositions are left regular compositions.
This motivates us to introduce    the notion of non-associative categories and functors between non-associative categories and establish some results on $\mathcal{LO}$ similar  to the classical case.

\subsection{Non-associative categories and weak functors}

For the sake of completeness,
we give the definition of non-associative categories and functors in this subsection.

\begin{mydef}
	A \textbf{non-associative category} $\mathcal{C}$ consists of three ingredients: a class $\obj{C}$ of
	objects, a set of morphisms $\Hom_{\mathcal{C}}(X,Y)$ for every ordered pair $(X,Y)$ of objects, and composition
	\begin{eqnarray*}
		\Hom_{\mathcal{C}}(X,Y)\times\Hom_{\mathcal{C}}(Y,Z)&\to& \Hom_{\mathcal{C}}(X,Z)\\
		(f,g) &\mapsto& gf,
	\end{eqnarray*}
	for every ordered triple $X,Y,Z$ of objects. These ingredients are subject to the
	following axioms:
	\begin{enumerate}
		\item 
		Each $f\in \Hom_{\mathcal{C}}(X,Y)$ has a
		unique domain $X$ and a unique target $Y$;
		\item For each object $X$, there is an identity morphism $1_X\in \Hom_{\mathcal{C}}(X,X)$ such
		that $$f\,1_X = f, \qquad 1_Y\, f = f$$ for all $f\in \Hom_{\mathcal{C}}(X,Y)$.
	\end{enumerate}
	\end{mydef}

	\begin{mydef}
		Let $\mathcal{C}$ be a non-associative category. A morphism $f\in \Hom_{\mathcal{C}}(X,Y)$ is called an \textbf{associative morphism} if it satisfies:
		\begin{numcases}{}
		(fg)h=f(gh),\notag &$\forall X''\stackrel{h}{\to}{X'}\stackrel{g}{\to}{X}\stackrel{f}{\to}{Y}$;\\
		(gf)h=g(fh), \notag &$\forall X'\stackrel{g}{\to}{X}\stackrel{f}{\to}{Y}\stackrel{h}{\to}{X''}$;\\
		 (gh)f=g(hf),\notag &$\forall X\stackrel{f}{\to}{Y}\stackrel{g}{\to}{X'}\stackrel{h}{\to}{X''}$.
		\end{numcases}
	Associated to the non-associative category $\mathcal{C}$ is
 an associative  category $\huaa{\mathcal{C}}$.
 In $\huaa{\mathcal{C}}$, the objects coincide with $\mathcal{C}$, the morphisms are associative morphisms from category $\mathcal{C}$,  and composition is inherited  from $\mathcal{C}$. We call $\huaa{\mathcal{C}}$ the \textbf{nucleus subcategory} of ${\mathcal{C}}$.
	
\end{mydef}

\begin{eg}\label{eg:lo,ro}
	We define a  category $\mathcal{LO}$. Its objects   are $\O$-bimodules, morphisms are left \almost linear maps, and compositions are left regular compositions.
	Clearly,  $\mathcal{LO}$ is non-associative.
we claim that  the nucleus subcategory of $\mathcal{LO}$ is the associative  category of $\O$-bimodules with $\O$-linear morphisms and usual compositions.
	
 Indeed, it follows from Lemma \ref{lem:[fgh]=} that every $\O$-linear map between $\O$-bimodules is  an associative morphism. Conversely, if $f\in \Hom_{\mathcal{LO}}(X,Y)$ is an associative morphism, then for any $p,q\in \O$, by definition we have
		\begin{eqnarray}\label{eqpf:fpq=fp q}
		(f\circledcirc R_q)\circledcirc R_p=f\circledcirc (R_q\circledcirc R_p).
		\end{eqnarray}
	It can be checked that $$R_q\circledcirc R_p=R_{pq}.$$
	Combining this with Lemma \ref{lem:Rp prop 2} and \eqref{eqpf:fpq=fp q}, we conclude that
	$$p\odot(q\odot f)=(pq)\odot f$$ for all $p,q\in \O$. This means that $$f\in \huaa{\Hom_{\mathcal{LO}}(X,Y)}=\Hom_{\O}(X,Y)$$ by Theorem \ref{thm:Hom kelie left,bi}. This proves the claim.

Similarly we can also define a category $\mathcal{RO}$. The objects  are also $\O$-bimodules, morphisms are right \almost linear maps, and compositions are  right regular compositions.
Similar result holds for category $\mathcal{RO}$.

\end{eg}

We introduce the notion of functors and weak functors between non-associative categories.
\begin{mydef}
	Let $\mathcal{C}$ and $\mathcal{D}$ be two non-associative categories.
	\begin{itemize}
		\item
	 A \textbf{covariant functor} $T: \mathcal{C}\to  \mathcal{D}$
	consists	of data:
	\begin{enumerate}
		\item for each $X\in \obj{C}$, a function
		\begin{eqnarray*}\obj{C}&\to&\obj{D}
			\\
			X&\mapsto& T(X);
		\end{eqnarray*}
		\item for each $X,\,X'\in \obj{C}$, a function
		\begin{eqnarray*}\Hom_{\mathcal{C}}(X,X') &\to& \Hom_{\mathcal{D}}(T(X),T(X'))
			\\
			f &\mapsto& T(f)
		\end{eqnarray*}
	\end{enumerate}
	satisfying the following axioms:
	\begin{enumerate}
			\item for any $X\stackrel{f}{\to}{X'}\stackrel{g}{\to}{X''}$ in $\mathcal{C}$,  		$T(gf)=T(g)T(f)$ in $\mathcal{D}$,
		\item $T(1_X)=1_{T(X)}$ for every $X\in \obj{C}$.
	\end{enumerate}

	\item
	A \textbf{covariant weak  functor} $T: \mathcal{C}\to  \mathcal{D}$
	consists	of data:
	\begin{enumerate}
		\item for each $X\in \obj{C}$, a function
		\begin{eqnarray*}\obj{C}&\to&\obj{D}
			\\
			X&\mapsto& T(X);
		\end{eqnarray*}
		\item for each $X,\,X'\in \obj{C}$, a function
		\begin{eqnarray*}\Hom_{\mathcal{C}}(X,X') &\to& \Hom_{\mathcal{D}}(T(X),T(X'))
			\\
			f &\mapsto& T(f)
		\end{eqnarray*}
	\end{enumerate}
	satisfying the following axioms:
	\begin{enumerate}
			\item for any $X\stackrel{f}{\to}{X'}\stackrel{g}{\to}{X''}$ in $\mathcal{C}$, if one of $f$ and $g$ is an associative morphism, then 		$T(gf)=T(g)T(f)$ in $\mathcal{D}$,
		\item $T(1_X)=1_{T(X)}$ for every $X\in \obj{C}$.
	\end{enumerate}
	
	\end{itemize}	
\end{mydef}

Similar definitions for contravariant functors and weak contravariant functors.
\begin{mydef}
	Let  $\mathcal{C}$ and $\mathcal{D}$ be two non-associative categories.
	\begin{itemize}
		\item
	
	A \textbf{contravariant functor} $T: \mathcal{C}\to  \mathcal{D}$ 	consists	of data:
	\begin{enumerate}
		\item for each $X\in \obj{C}$, a function
		\begin{eqnarray*}\obj{C}&\to&\obj{D}
			\\
			X &\mapsto& T(X);
		\end{eqnarray*}
		\item for each $X,\,X'\in \obj{C}$, a function
		\begin{eqnarray*} \Hom_{\mathcal{C}}(X,X')&\to&\Hom_{\mathcal{D}}(T(X'),T(X))
			\\ f&\mapsto& T(f)
		\end{eqnarray*}
	\end{enumerate}
	satisfying the following axioms:
	\begin{enumerate}
			\item for any $X\stackrel{f}{\to}{X'}\stackrel{g}{\to}{X''}$ in $\mathcal{C}$,  		$T(gf)=T(f)T(g)$ in $\mathcal{D}$,
		\item $T(1_X)=1_{T(X)}$ for every $X\in \obj{C}$.
	\end{enumerate}
\item
A \textbf{contravariant weak  functor} $T: \mathcal{C}\to  \mathcal{D}$ 	consists	of data:
\begin{enumerate}
	\item for each $X\in \obj{C}$, a function
	\begin{eqnarray*}\obj{C}&\to&\obj{D}
		\\
		X &\mapsto& T(X);
	\end{eqnarray*}
	\item for each $X,\,X'\in \obj{C}$, a function
	\begin{eqnarray*} \Hom_{\mathcal{C}}(X,X')&\to&\Hom_{\mathcal{D}}(T(X'),T(X))
		\\ f&\mapsto& T(f)
	\end{eqnarray*}
\end{enumerate}
satisfying the following axioms:
\begin{enumerate}
		\item for any $X\stackrel{f}{\to}{X'}\stackrel{g}{\to}{X''}$ in $\mathcal{C}$, if one of $f$ and $g$ is an associative morphism, then $T(gf)=T(f)T(g)$ in $\mathcal{D}$,
	\item $T(1_X)=1_{T(X)}$ for every $X\in \obj{C}$.
\end{enumerate}
\end{itemize}
\end{mydef}	
\begin{rem}
By definition, we know that every functor is automatically a weak functor and the two notions of  weak functor  and functor coincide in the case of associative categories.
Generally, the composition  of two weak functors may  be not a weak functor any more. However, it is easy to see that the composition  of a weak functor and a functor  continues to be a weak functor.
\end{rem}
\begin{eg}\label{eg:w-hom}
	We denote by \text{\bf{Sets}} the associative category of sets. Let $\mathcal{C}$ be a non-associative category and $B\in \obj{\mathcal{C}}$.  Then there is a  natural covariant \textbf{weak Hom functor}
	$$w-\Hom_\mathcal{C}(B,-):\mathcal{C}\to \text{\bf{Sets}}$$
	and a contravariant	\textbf{weak Hom functor} $$w-\Hom_\mathcal{C}(-,B):\mathcal{C}\to \text{\bf{Sets}}.$$
	
	The covariant	\textbf{weak Hom functor} is defined, for
	all $C \in \obj{\mathcal{C}}$, by
	$$w-\Hom_\mathcal{C}(B,-)(C)=\Hom_\mathcal{C}(B,C)$$
	and if $f: C \to C'\in \mathcal{C}$, then $w-\Hom_\mathcal{C}(B,-)(f): \Hom_\mathcal{C}(B,C) \to\Hom_\mathcal{C}(B,C')$ is given
	by
	$$w-\Hom_\mathcal{C}(B,-)(f): h\mapsto fh.$$
	
	The contravariant	\textbf{weak Hom functor} is defined, for
	all $C \in \obj{\mathcal{C}}$, by
	$$w-\Hom_\mathcal{C}(-,B)(C)=\Hom_\mathcal{C}(C,B)$$
	and if $f: C \to C'\in \mathcal{C}$, then $w-\Hom_\mathcal{C}(-,B)(f): \Hom_\mathcal{C}(C',B) \to\Hom_\mathcal{C}(C,B)$ is given
	by
	$$w-\Hom_\mathcal{C}(-,B)(f): h\mapsto hf.$$
	
	We emphasize that the \textbf{Hom functors} defined above are all weak functors rather than genuine functors. This is because that the category \text{\bf{Sets}} is  associative. For the special case of non-associative categories of $\mathcal{LO}$ and $\mathcal{RO}$, we shall introduce  \textbf{Hom functors} whose target categories are non-associative in Subsection \ref{sec:hom}.
\end{eg}
\begin{eg}\label{eg:o Mod R}
We denote by $\O\text{-}\textbf{Mod}_\R$ the associative category whose objects are $\O$-bimodules and morphisms are $\R$-linear maps and compositions are usual compositions. Note that the associative morphisms in $\mathcal{LO}$ are $\O$-linear maps.	Then in view of Lemma \ref{lem:vanishing[f,g,x]=0}, we have a covariant weak functor
\begin{align*}
\iota:\qquad\quad\mathcal{LO}\quad&\xlongrightarrow[]{\hskip1cm} \quad\O\text{-}\textbf{Mod}_\R\\
M\quad&\shortmid\!\xlongrightarrow[]{\hskip1cm}\quad M\\
f:M\to N\quad&\shortmid\!\xlongrightarrow[]{\hskip1cm}\quad f.
\end{align*}
\end{eg}

\begin{eg}\label{eg:forget}
	Let $$U :\qquad\quad \mathcal{LO}\quad\xlongrightarrow[]{\hskip1cm} \quad \text{\bf{Sets}}$$ be the forgetful functor which assigns to each
	$\O$-bimodule $M$ its underlying set and views each left para-linear map as a mere
	function. Then by the similar argument, we conclude that $U$ is covariant weak functor.
\end{eg}

\begin{eg}
	Let $\R\text{-}\textbf{Vec}$ denote the category of real vector spaces.
	We define a weak functor
	\begin{align*}
	\mathfrak{Re}:\qquad\quad\mathcal{LO}\quad&\xlongrightarrow[]{\hskip1cm} \quad\R\text{-}\textbf{Vec}\\
	M\quad&\shortmid\!\xlongrightarrow[]{\hskip1cm}\quad \re M\\
	f:M\to N\quad&\shortmid\!\xlongrightarrow[]{\hskip1cm}\quad (\re f)|_{\re M}.
	\end{align*}
	By \eqref{eq:re f=ext( fR| reM)}, we know that
	$$\mathfrak{Re}(f)=(\re f)|_{\re M}=f_{\R}|_{\re M}:\re M\to \re M'.$$
	This means for any $x\in \re M$, we have
	\begin{eqnarray}
	(\mathfrak{Re}(f))(x)=f_{\R}(x)=\re f(x).
	\end{eqnarray}
	Let  $X\stackrel{f}{\to}{X'}\stackrel{g}{\to}{X''}$. Then for all $x\in \re X$,
	$$(\mathfrak{Re}(g\circledcirc f))(x)=\re ((g\circledcirc f)(x))$$ and
	 $$(\mathfrak{Re}(g))\circ(\mathfrak{Re}(f))(x)=\re g(\re f(x)).$$
	If one of $f$ and $g$ is $\O$-linear, then it follows from Lemma \ref{lem:vanishing[f,g,x]=0} that
	$$\re ((g\circledcirc f)(x))=\re ((g\circ  f)(x)).$$
If $f$ is $\O$-linear, then $$(\mathfrak{Re}(g))\circ(\mathfrak{Re}(f))(x)=\re g(f(x)).$$
If $g$ is $\O$-linear, then it follows from Lemma \eqref{lem:f re=re f} that
$$\re ((g\circ  f)(x))=g(\re f(x))=\re g(\re f(x)).$$ In both cases, we  obtain that $$(\mathfrak{Re}(g))\circ(\mathfrak{Re}(f))=\re ((g\circledcirc f).$$ This proves that $\mathfrak{Re}$ is a weak functor.
\end{eg}

\begin{mydef}
	Let  $\mathcal{C}$, $\mathcal{D}$ be two non-associative categories.
	Suppose   $S,T: \mathcal{C} \rightarrow \mathcal{D}$  are two covariant weak  functors. A \textbf{natural  transformation}  $\tau: S \rightarrow T $ is a one-parameter family of morphisms in  $\mathcal{{D}}$, 	
	$$\tau = \left(\tau_{X}: S (X) \rightarrow T (X)\right)_{X \in \obj{\mathcal{C}}}$$		making the following diagram
	commute in $\mathcal{{D}}$:
	$$\begin{array}[c]{ccc}
	S(X)\quad&\stackrel{\tau_{X}}{\xlongrightarrow[]{\hskip1cm}}&T(X)\\
	\, \Biggm\downarrow\scriptstyle{S(f)}&&\Biggm\downarrow\scriptstyle{ T(f)}\\
	S(Y)\quad&\stackrel{\tau_{Y}}{\xlongrightarrow[]{\hskip1cm}}&T(Y).
	\end{array}$$
	for all $ X \stackrel{f}{\longrightarrow} Y $ in  ${\mathcal{C}}$.

 	A \textbf{natural isomorphism} is a natural  transformation $\tau$ for
	which each $\tau_X$ is an isomorphism.
	Natural  transformation  between contravariant  functors can be defined similarly.
\end{mydef}
\begin{rem}
	We point out that the usual composition  of two  natural transformations may be not a  natural transformation generally. 
\end{rem}

We make the following notation. If $F, G : \mathcal{C} \to \mathcal{{D}}$ are weak functors of the same variance, then we denote
$$\text{Nat}(F, G) :=\{\text{natural transformations } F\to G\}.$$
\begin{thm}[\textbf{Yoneda lemma}]
	Let $\mathcal{C}$ be a category, let $A\in  \obj{\mathcal{C}}$, and
	let $G:\mathcal{C}\to  \text{\bf{Sets}}$ be a covariant  functor. Then there is an bijection
	\begin{eqnarray*}
		y:\text{Nat}(w-\Hom_{\mathcal{C}}(A,-), G)& \to& G(A)\\
		\tau&\mapsto&\tau_A(1_A).
	\end{eqnarray*}
\end{thm}
\begin{proof}
	If $\tau : w-\Hom_{\mathcal{C}}(A,-)\to G$ is a  natural transformation, then $
	\tau_A(1_A)$ lies in the set $G(A)$, for $\tau_A :\Hom_{\mathcal{C}}(A,A) \to G(A).$ Thus,  $y$ is
	well-defined.
	
	We prove that $y$ is an injection.
	For each $B\in \obj{\mathcal{C}}$ and $f\in \Hom_{\mathcal{C}}(A,B)$, there is a commutative diagram
	$$\begin{array}[c]{ccc}
\Hom_{\mathcal{C}}(A,A)\quad&\stackrel{\tau_{A}}{\xlongrightarrow[]{\hskip1cm}}&G(A)\\
	\quad \Biggm\downarrow\scriptstyle{w-\Hom_{\mathcal{C}}(A,-)(f)}&&\Biggm\downarrow\scriptstyle{ G(f)}\\
	\Hom_{\mathcal{C}}(A,B)\quad&\stackrel{\tau_{B}}{\xlongrightarrow[]{\hskip1cm}}&G(B),
	\end{array}$$
	so that $$(G(f))\tau_A(1_A) = \tau_B(w-\Hom_{\mathcal{C}}(A,-)(f)(1_A)) = \tau_B(f1_A) = \tau_B(f).$$
	If $\sigma : w-\Hom_{\mathcal{C}}(A,-)\to G$ is  another  natural transformation, then $$\sigma_B(f) =
	(G(f))\sigma_A(1_A).$$ Hence, if $\sigma_A(1_A) = \tau_A(1_A)$, then $\sigma_B = \tau_B$ for all $B \in \obj{\mathcal{C}}$
	and, hence, $\sigma = \tau$. Therefore, $y$ is an injection.
	
	To see that $y$ is a surjection, take $x\in  G(A)$. For $B\in \obj{\mathcal{C}}$ and $g\in	\Hom_\mathcal{C}(A, B)$, define
	$$\tau_B(g) = (G(g))(x).$$
	We claim that $\tau$ is
	a   natural transformation; that is, if $h : B \to C$ is a morphism in $\mathcal{C}$, then we have $$(G(h)\circ\tau_B)(g) = G(h)\circ G(g)(x)$$ and $$\tau_C(w-\Hom_{\mathcal{C}}(A,-)(h)(g))=\tau_C(hg)=G(hg)(x).$$
	Since $G$ is a covariant  functor, it follows that
	$$ G(h)\circ G(g)=G(hg).$$
	This implies that $\tau$ is a  natural transformation and $$y(\tau) =
	\tau_A(1_A) = G(1_A)(x) = x.$$ Hence $y$ is a surjection. This completes the proof.
\end{proof}

\begin{rem}
If $G$ is a weak functor, then we shall only get an injection in the Yoneda lemma.
\end{rem}
These discussions show that the notion of weak functors in the setting of non-associative categories is natural. However, to  study   functors between non-associative categories is   difficult and important. We shall discuss functors on the specific non-associative category $\mathcal{LO}$ in  the following subsections.
\subsection{Conjugate functor}

The duality   in subsection \ref{sec:right O mod} can be interpreted as a covariant functor between two non-associative categories $\mathcal{LO}$ and $\mathcal{RO}$.
We call it   the \textbf{conjugate functor}. It is defined by
\begin{align*}
C:\qquad\quad\mathcal{LO}\quad&\xlongrightarrow[]{\hskip1cm} \quad\mathcal{RO}\\
M\quad&\shortmid\!\xlongrightarrow[]{\hskip1cm}\quad M^{C}\\
f\in \Hom_\mathcal{LO}(M,N)\quad&\shortmid\!\xlongrightarrow[]{\hskip1cm}\quad C(f):=f^{C}.
\end{align*}
We recall the notation $M^{C}$ of an $\O$-bimodule $M$ and the notation $f^{C}$ in subsection \ref{sec:right O mod}, where the bimodule structure on $M^{C}$ is defined as:
$$p\cdot x=x\overline{p},\quad x\cdot p=\overline{p}x.$$
For each $f\in \Hom_\mathcal{LO}(M,N)$, the morphism $C(f)$ is just defined to be $f$ as a map of sets (which is denoted by $f^{C}$ in subsection \ref{sec:right O mod}).

\begin{thm}
	The conjugate functor $C:\mathcal{LO}\longrightarrow \mathcal{RO}$ is a covariant functor.
\end{thm}
\begin{proof}
	
	We need to verify two assertions.
	
	\begin{itemize}\item Assertion 1:
		$$\text{For any }f\in \Hom_\mathcal{LO}(M,N),  \quad \mbox{ we have }C(f)\in \Hom_\mathcal{RO}(M^{C},N^{C}). $$

		In fact, {we have proved this in subsection \ref{sec:right O mod}. More precisely, we have }
		$$\re B_p(C(f),x)=\re A_p(x,f)=0,$$
		{which means that $ C(f)\in \Hom_\mathcal{RO}(M^{C},N^{C})$.}
		\item Assertion 2:
		$$C(f\circledcirc_l g)=C(f)\circledcirc_r C(g).$$ Here the subscripts indicate  which type of the regular composition is used.

		Indeed, in view of Remark \ref{rem:[x,f,g]=[f,g,x]}  we have
		\begin{eqnarray*}C(f\circledcirc_l g)(x)&=&(f\circledcirc_l g)(x)
			\\
			&=&f\circ g (x)+[f,g,x]
			\\
			&=&f^{C}\circ g^{C} (x)+[x,f^{C},g^{C}]
			\\ &=&(C(f)\circledcirc_r C(g))(x).
		\end{eqnarray*}
		
	\end{itemize}
	This proves that $C$ is a covariant functor.	
\end{proof}

Similarly, we also  have a covariant functor from $\mathcal{RO}$ to $\mathcal{LO}$, {also denoted by $C$},
\begin{align*}
C:\qquad\quad\mathcal{RO}\quad&\xlongrightarrow[]{\hskip1cm} \quad\mathcal{LO}\\
M\quad&\shortmid\!\xlongrightarrow[]{\hskip1cm}\quad M^{C}\\
f\in \Hom_\mathcal{RO}(M,N)\quad&\shortmid\!\xlongrightarrow[]{\hskip1cm}\quad C(f):=f^{C}.
\end{align*}
For each $f\in \Hom_\mathcal{RO}(M,N)$, the morphism $C(f)$ is just defined to be $f$ as a map of sets.

Obviously, $C$ is an isomorphism between  categories $\mathcal{LO}$ and $\mathcal{RO}$.

\subsection{Hom functors}\label{sec:hom}

We  introduce
the \textbf{Hom functors} between the non-associative  categories $\mathcal{LO}$ and $\mathcal{RO}$.
\begin{thm}\label{thm:Hom functor}
	For any $\O$-bimodule $M$, we have a covariant functor $$\Hom_\mathcal{LO}(M,-):\mathcal{LO}\to\mathcal{LO}
	$$ and a contravariant functor $$\Hom_\mathcal{LO}(-,M):\mathcal{LO}\to\mathcal{RO}.$$
	Moreover, both functors are $\O$-linear.
\end{thm}
\begin{proof} For short,
	we denote the two functors above as
	\begin{eqnarray*}
		T &=& \Hom_\mathcal{LO}(M,-),
		\\
		S &=& \Hom_\mathcal{LO}(-,M).
	\end{eqnarray*}
	The definition of $T$  and $S$ is canonical. More precisely,
	\begin{eqnarray*}T(X)&=&\Hom_\mathcal{LO}(M,X),\\  S(X)&=&\Hom_\mathcal{LO}(X,M)
	\end{eqnarray*}
	for any $X\in\obj{LO}$. For each morphism $f\in \Hom_\mathcal{LO}(X,Y),$ we define for every $g\in \Hom_\mathcal{LO}(M,X)$
	\begin{eqnarray*}\fx{g}{T(f)}&:=&f\circledcirc_l g,\\
	\end{eqnarray*}
	and  for every $g\in \Hom_\mathcal{LO}(Y,M)$
	\begin{eqnarray*}
		\fx{S(f)}{g}&:=&g\circledcirc_l f.
	\end{eqnarray*}
	For simplicity, we omit the subscript `l' since the regular composition is understood well.	
	
	We  show that $T$ is a covariant functor.
	For any $p\in \O$,  $f\in \Hom_\mathcal{LO}(X,Y)$, and $g\in  \Hom_\mathcal{LO}(M,X)$,  it follows from  Lemma \ref{lem:Rp prop 2} that
	\begin{align*}
	A_p(g,T(f))&=f\circledcirc (p\odot g)-p\odot(f\circledcirc g)\\
	&=f\circledcirc(g\circledcirc R_p)-(f\circledcirc g)\circledcirc R_p\\
	&=-[f,g,R_p].
	\end{align*}
	In view of  Lemma \ref{lem:[fgh]=}, we have
	$$\re A_p(g,T(f))=0,$$
	which means $$T(f)\in \Hom_\mathcal{LO}(T(X),T(Y)).$$

	For a pair of morphisms   $$X\stackrel{f_2}{\longrightarrow}{X'}\stackrel{f_1}{\longrightarrow}{X''} $$
	in $\mathcal{LO}$, 		we need to verify
	\begin{eqnarray}\label{pfeq:T(fg)=T(f)T(g))}
	T(f_1\circledcirc f_2)=T(f_1)\circledcirc T(f_2).
	\end{eqnarray}
	For any $$g\in\Hom_\O(M,X),$$ we have
	\begin{align*}
	&{\ \ \ \,}\fx{g}{	T(f_1\circledcirc f_2)-T(f_1)\circledcirc T(f_2)}\\
	&=(f_1\circledcirc f_2)\circledcirc g-\big(T(f_1)\circ T(f_2)\big)(g)-[T(f_1),T(f_2),g]&\text{using Lemma \ref{lem:vanishing[f,g,x]=0}}\\
	&=(f_1\circledcirc f_2)\circledcirc g-T(f_1)(f_2\circledcirc g) \\
	&=(f_1\circledcirc f_2)\circledcirc g-f_1\circledcirc(f_2\circledcirc g)\\
	&=[f_1,f_2,g]&\text{using Lemma \ref{lem:[fgh]=}}\\
	&=0.
	\end{align*}
	Notice that  $$\re \Hom_\mathcal{LO}(M,X)=\Hom_\O(M,X)$$
	due to Theorem \ref{thm:Hom kelie left,bi}.
	Since both sides of \eqref{pfeq:T(fg)=T(f)T(g))} are left \almost linear maps,   it follows from Corollary \ref{cor: A_p(x,f)=0 and f(px)=pf(x)} that   \eqref{pfeq:T(fg)=T(f)T(g))} holds.	
	
	Next we come to show that $T$ is an $\O$-linear functor, i.e.,  for any $p\in \O$ and  $f\in \Hom_\mathcal{LO}(X,Y)$,
	\begin{equation*}
	T(p\odot f)=p\odot T(f).
	\end{equation*}
	If $g\in  \re \Hom_\mathcal{LO}(M,X)=\Hom_\O(M,X)$ (by Theorem \ref{thm:Hom kelie left,bi}), then again  from Lemma \ref{lem:[fgh]=}  we have
	\begin{eqnarray*}
		\fx{g}{T(p\odot f)-p\odot T(f)}&=&(p\odot f)\circledcirc g-f\circledcirc (g\odot p)\\
		&=&(f\circledcirc R_p)\circledcirc g-f\circledcirc (R_p\circledcirc g)\\
		&=&[f,R_p,g]\\
		&=&0.
	\end{eqnarray*}
	It follows from Corollary \ref{cor: A_p(x,f)=0 and f(px)=pf(x)}  that $T$ is $\O$-linear.
	
	Secondly, we shall  prove that $S$ is a contravariant functor.
	For any $p\in \O$,  $f\in \Hom_\mathcal{LO}(X,Y)$, and $g\in  \Hom_\mathcal{LO}(Y,M)$,  it follows from  Lemma  \ref{lem:Rp prop1} that 	
	\begin{align*}
	B_p(S(f),g)&=(g\circledcirc f)\odot p-(g\odot p)\circledcirc f\\
	&=R_p\circledcirc(g\circledcirc f)-(R_p\circledcirc g)\circledcirc f\\
	&=-[R_p,g,f].
	\end{align*}
	Hence by Lemma \ref{lem:[fgh]=}, we have $$\re B_p(S(f),g)=0,$$
	which means
	$$S(f)\in \Hom_\mathcal{RO}(S(Y),S(X)).$$

	For a pair of morphisms $$X\stackrel{f_2}{\longrightarrow}{X'}\stackrel{f_1}{\longrightarrow}{X''}$$ in $\mathcal{LO}$, we need to show
	$$S(f_1\circledcirc f_2)=S(f_2)\circledcirc S(f_1).$$
	In fact, for any $$g\in \re \Hom_\mathcal{LO}(Y,M)=\Hom_\O(Y,M),$$ we have
	\begin{align*}
	\fx{	S(f_1\circledcirc f_2)-S(f_2)\circledcirc S(f_1)}{g}
	&=g\circledcirc (f_1\circledcirc f_2)-(g\circledcirc f_1 )\circledcirc f_2\\
	&=-[g,f_1,f_2]\\
	&=0.
	\end{align*}
	Then it follows from Corollary \ref{cor: A_p(x,f)=0 and f(px)=pf(x)} that $S$ is a contravariant functor.
	
	Finally, it remains
	to show that $S$ is an $\O$-linear functor, i.e., for any $p\in \O$ and  $f\in \Hom_\mathcal{LO}(X,Y)$,
	\begin{equation*}
	S(p\odot f)=p\odot S(f).
	\end{equation*}
	Let $g\in  \re \Hom_\mathcal{LO}(M,Y)$. Then
	\begin{eqnarray*}
		\fx{S(p\odot f)-p\odot S(f)}{g}&=&g\circledcirc (p\odot f)-p\odot (g\circledcirc f)\\
		&=&g\circledcirc (f\circledcirc R_p)- (g\circledcirc f)\circledcirc R_p\\
		&=&-[g,f,R_p]\\
		&=&0.
	\end{eqnarray*}
	It follows from Corollary \ref{cor: A_p(x,f)=0 and f(px)=pf(x)} again that $S$ is $\O$-linear.
	This completes the proof.
\end{proof}

\begin{rem}
	Similarly, we also have a covariant functor $$\Hom_\mathcal{RO}(M,-):\mathcal{RO}\to\mathcal{RO}
	$$ and a contravariant functor $$\Hom_\mathcal{RO}(-,M):\mathcal{RO}\to\mathcal{LO}.$$

	In contrast to the classical associative case where the target category of a Hom functor is always the \text{\bf{Sets}} category, in the non-associative case the Hom functors admit a non-associative  target category. Moreover, the weak Hom functors introduced in Example \ref{eg:w-hom} have a close relation with Hom functors. For example, we have
	$$w-\Hom_\mathcal{LO}(M,-)=U\circ \Hom_\mathcal{LO}(M,-),$$
	where $U$ is the forgetful functor introduced in Example \ref{eg:forget}. This provides an alternative viewpoint to see that $w-\Hom_\mathcal{LO}(M,-)$ is a weak functor.

	We shall always denote $$f_{*_l}:=\Hom_\mathcal{LO}(M,-)(f):\Hom_\mathcal{LO}(M,X)\to \Hom_\mathcal{LO}(M,Y),$$
	$$ f^{*_l}:=\Hom_\mathcal{LO}(-,M)(f):\Hom_\mathcal{LO}(Y,M)\to \Hom_\mathcal{LO}(X,M)$$ for any  $X\stackrel{f}{\longrightarrow}{Y}$ in  $\mathcal{LO}$.
	Similar notations can be defined   for $\mathcal{RO}$.

	{Theorem \ref{thm:Hom functor} provides  a fundamental fact in the study of octonionic functional analysis.  The Hom functors can be used to define the Banach dual operator and Hilbert dual operator. This answers that why  the  seemly right method to define the Hilbert dual of an operator $T$ as
		$$\fx{Tx}{y}=\fx{x}{Ty}$$ does not work in an octonionic Hilbert space \cite{goldstine1964hilbert}.  As a direct consequence of Theorem \ref{thm:Hom functor}, we have    the basic  properties of  dual operators:
		\begin{eqnarray*}
			&&(p\odot f)^{*_l}=p\odot f^{*_l},
			\\ && (f\odot p)^{*_l}=f^{*_l}\odot p,
			\\
			&&(f\circledcirc g)^{*_l}=g^{*_l}\circledcirc f^{*_l}
		\end{eqnarray*}
		for any $p\in \O$ and $X\stackrel{g}{\longrightarrow}{Y}\stackrel{f}{\longrightarrow}{Z}$.
	}
\end{rem}

As a consequence, we have a covariant functor $$\Hom_\mathcal{RO}(\Hom_\mathcal{LO}(-,\O),\O):\mathcal{LO}\to \mathcal{RO}\to \mathcal{LO}.$$
We denote by $1_{\mathcal{LO}}$ the identity functor. Then the embedding $\O$-homomorphism $$\tau_X:X\to X^{**}$$
introduced in Subsection \ref{sec:semi-reflex}  is a natural transformation of functors. This can be seen from the following lemma.

\begin{lemma}
	Let $X$ and $Y$ be two   $\spo$-bimodules. Then for any   $\phi\in \Hom_\mathcal{LO}(X,Y)$, we have
	\begin{eqnarray}\label{eq:T''x''=(Tx)''}
	\phi^{*_l*_r}\circledcirc_l\tau_X=\tau_Y\circledcirc_l \phi.
	\end{eqnarray}  
\end{lemma}
\begin{proof}
	Note that we have shown in Subsection \ref{sec:semi-reflex} that $\tau_X $
is an $\O$-linear map. Hence \eqref{eq:T''x''=(Tx)''} is equivalent to $$\phi^{*_l*_r}\circ\tau_X=\tau_Y\circ \phi.$$

	 For any $x\in X$ and any $f\in Y^{*_l}$, we  denote $x''=\tau_X (x)$ as usual, then
	\begin{eqnarray*}
		\fx{\phi^{*_l*_r}(x'')}{f}&=&(x''\circledcirc_r \phi^{*_l})(f)\\
		&=&x''(\phi^{*_l}(f))+[f,x'',\phi^{*_l}]\\
		&=&\fx{x}{\phi^{*_l}(f)}+[f,x'',\phi^{*_l}].
	\end{eqnarray*}
	By Definition \ref{def:right mod regular composition}, we get
	\begin{eqnarray*}
		[f,	x'',\phi^{*_l}]&=&\sum_{j=1}^7\big(\re \fx{x''}{B_{e_j}(\phi^{*_l},f)}\big)e_j\\
		&=&\sum_{j=1}^7 \big(\re \fx{x}{B_{e_j}(\phi^{*_l},f)}\big)e_j.
	\end{eqnarray*}
It follows from Lemma \ref{lem:Rp prop1} that for all $p\in \O$,
\begin{eqnarray*}
B_p(\phi^{*_l},f)&=&\phi^{*_l}(f)\odot p-\phi^{*_l}(f\odot p)\\
&=&(f\circledcirc \phi)\odot p-(f\odot p)\circledcirc \phi\\
&=&R_p\circledcirc(f\circledcirc \phi)-(R_p\circledcirc f)\circledcirc \phi\\
&=&-[R_p,f,\phi].
\end{eqnarray*}
In view of
\eqref{eq:[f,g,h]=}, we thus obtain
\begin{eqnarray*}
	\sum_{j=1}^7 \big(\re \fx{x}{B_{e_j}(\phi^{*_l},f)}\big)e_j&=&-\sum_{j=1}^7 \big(\re \fx{x}{[R_{e_j},f,\phi]}\big)e_j\\
	&=&\sum_{j=1}^7 \big(\re R_{e_j}([f,\phi,x])\big)e_j\\
	&=&-[f,\phi,x].
\end{eqnarray*}
	We  used the fact that $[f,\phi,x]\in \pureim{\spo}$  in the last equality.
	Therefore  $\phi^{*_l*_r}(x'')=(\phi (x))''$ for all $x\in X$. That is, $\phi^{*_l*_r}\circ\tau_X=\tau_Y\circ \phi$.
\end{proof}

\subsection{Tensor functors}
In this subsection, we introduce the tensor functors between  non-associative categories $\mathcal{LO}$ and $\mathcal{RO}$.

 We define
the  tensor product of two $\O$-bimodules $M$ and  $M'$ as
\begin{eqnarray}\label{eqdef:MotimesO M'}
M\otimes_\O M':=(\re M\otimes_\R \re M')\otimes_\R\O.
\end{eqnarray}
Then $M\otimes_\O M'$ is an $\O$-bimodule and
$$\re (M\otimes_\O M')= \re M\otimes_\R \re M'.$$
For  $m\in M$ and $m'\in M'$, we write $$m=\sum_{i=0}^7e_im_i,\qquad m'=\sum_{i=0}^7e_im'_i,$$ where $m_i\in \re M$ and $m'_i\in \re M'$. We define
\begin{eqnarray*}
	m\otimes_\O m':=\sum_{i,j=0}^7(m_i\otimes_\R m'_j)\otimes_\R (e_ie_j).
\end{eqnarray*}
We remark that in general for any $p\in \O$, $$mp\otimes_\O m'\neq m\otimes_\O pm'.$$
Instead, by direct calculation, we have the following identity
\begin{eqnarray}\label{eq:mpotm'}
mp\otimes_\O m'-m\otimes_\O pm'=\sum_{i,j=0}^7(m_i\otimes_\R m'_j)\otimes_\R [e_i,p,e_j].
\end{eqnarray}

We shall use the specific case of \eqref{eq:mpotm'} and its variants as follows.
\begin{lemma}\label{lem:tf}
	Let $M$,  $M'$ be two $\O$-bimodules and $p\in \O$, $m\in M$, $m'\in M'$.
	If one of the three elements $p$, $m$ and $m'$ is an associative element, then we have
	\begin{eqnarray}\label{eq:p(motimesO m')=pmotimesO m'}
	p(m\otimes_\O m')=pm\otimes_\O m',\label{eq:tf1}\\
	mp\otimes_\O m'=m\otimes_\O pm',\label{eq:tf2}\\
	(m\otimes_\O m')p=m\otimes_\O m'p.\label{eq:tf3}
	\end{eqnarray}
\end{lemma}
\begin{proof}
	Assertion \eqref{eq:tf2} follows from \eqref{eq:mpotm'} directly. The proof of \eqref{eq:tf1} and \eqref{eq:tf3} is similar.
\end{proof}

 Now we come to define  tensor functors.
Fix an $\O$-bimodule $M$.
For any morphism $X\stackrel{f}{\longrightarrow}X'$, we introduce a map
\begin{eqnarray*}
	f_M:\re M\otimes_\R \re X&\to&  M\otimes_\O  X'\\
	m\otimes_\R x &\mapsto&m\otimes_\O f(x).
\end{eqnarray*}
 By  the left \almost linear extension defined in Lemma \ref{lem:ext}, $f_M$ induces a left \almost linear map:
$$l\text{-}\ext f_M\in \Hom_\mathcal{LO}(M\otimes_\O X,M\otimes_\O X').$$
In view of identities \eqref{eq:tf1} and \eqref{eq:tf2}, by direct calculations we obtain that
\begin{eqnarray}\label{eq:tensor product}
l\text{-}\ext f_M(m\otimes_\O x)=m\otimes_\O f(x)
\end{eqnarray}
for all $m\in \re M$   and all $x\in X$.
We thus have a natural tensor functor
\begin{align*}
M\otimes^{ll}_\O-:\qquad\mathcal{LO}\qquad&\xlongrightarrow[]{\hskip1cm} \qquad\mathcal{LO}\\
X\qquad&\shortmid\!\xlongrightarrow[]{\hskip1cm} \qquad M\otimes_\O X\\
X\stackrel{f}{\to}X'\qquad&\shortmid\!\xlongrightarrow[]{\hskip1cm} \qquad l\text{-}\ext f_M.
\end{align*}

\begin{thm}
	For any $\O$-bimodule $M$, the tensor functor	$$M\otimes^{ll}_\O-:\mathcal{LO}\xlongrightarrow[]{\hskip1cm} \mathcal{LO} $$
	 is	a covariant functor.
\end{thm}

\begin{proof}

	To show $	M\otimes^{ll}_\O-$ is  a covariant functor, it only needs to verify	that
	\begin{eqnarray}\label{eqpf:tensor functor}
	l\text{-}\ext f_M \circledcirc l\text{-}\ext g_M=l\text{-}\ext (f\circledcirc g)_M
	\end{eqnarray}
	for any morphisms  $$X\stackrel{g}{\longrightarrow}X'\stackrel{f}{\longrightarrow}X''$$
	 in $\mathcal{LO}$.

	Let $m\in \re M$ and $x\in\re X $.
	It follows that $m\otimes_\O x\in \re (M\otimes_\O X)$.
	Then by Lemma \ref{lem:vanishing[f,g,x]=0}, we have
	\begin{eqnarray}\label{eqpf:tensor product}
	(l\text{-}\ext f_M \circledcirc l\text{-}\ext g_M)(m\otimes_\O x)&=&(l\text{-}\ext f_M \circ l\text{-}\ext g_M)(m\otimes_\R x)\notag\\
	&=&l\text{-}\ext f_M \, (m\otimes_\O g(x)).
	\end{eqnarray}
	Combining \eqref{eqpf:tensor product} with  \eqref{eq:tensor product}, we obtain
	\begin{eqnarray}\label{eqpf:tf l}
	(l\text{-}\ext f_M \circledcirc l\text{-}\ext g_M)(m\otimes_\O x)&=&m\otimes_\O f\circ g(x).
	\end{eqnarray}
	Using \eqref{eq:tensor product} again, we get
	\begin{eqnarray}\label{eqpf:tf r}
	l\text{-}\ext (f\circledcirc g)_M  (m\otimes_\O x)=m\otimes_\O (f\circledcirc g)(x)=m\otimes_\O f\circ g(x).
	\end{eqnarray}
	It follows from  \eqref{eqpf:tf l}and  \eqref{eqpf:tf r} that
	$$(l\text{-}\ext f_M \circledcirc l\text{-}\ext g_M)(m\otimes_\O x)=l\text{-}\ext (f\circledcirc g)_M  (m\otimes_\O x).$$
	Note that $$\re (M\otimes_\O X)= \re M\otimes_\R \re X.$$
	Hence \eqref{eqpf:tensor functor}  follows from Corollary \ref{lem: f(huaa M)=0 yields f=0}.
\end{proof}	

Similarly, we can also define three tensor functors $M\otimes^{lr}_\O-$, $M\otimes^{rr}_\O-$ and $M\otimes^{rl}_\O-$. They are defined as follows
\begin{align*}
M\otimes^{lr}_\O-:\qquad\mathcal{LO}\qquad&\xlongrightarrow[]{\hskip1cm}  \qquad\mathcal{RO}\\
X\qquad&\shortmid\!\xlongrightarrow[]{\hskip1cm}  \qquad M\otimes_\O X\\
X\stackrel{f}{\to}X'\qquad&\shortmid\!\xlongrightarrow[]{\hskip1cm}  \qquad r\text{-}\ext f_M;
\end{align*}
\begin{align*}
M\otimes^{rr}_\O-:\qquad\mathcal{RO}\qquad&\xlongrightarrow[]{\hskip1cm}  \qquad\mathcal{RO}\\
X\qquad&\shortmid\!\xlongrightarrow[]{\hskip1cm}  \qquad M\otimes_\O X\\
X\stackrel{f}{\to}X'\qquad&\shortmid\!\xlongrightarrow[]{\hskip1cm}  \qquad r\text{-}\ext f_M;
\end{align*}
\begin{align*}
M\otimes^{rl}_\O-:\qquad\mathcal{RO}\qquad&\xlongrightarrow[]{\hskip1cm}  \qquad\mathcal{LO}\\
X\qquad&\shortmid\!\xlongrightarrow[]{\hskip1cm}  \qquad M\otimes_\O X\\
X\stackrel{f}{\to}X'\qquad&\shortmid\!\xlongrightarrow[]{\hskip1cm}  \qquad l\text{-}\ext f_M.
\end{align*}

We shall denote $$(	M\otimes^{ll}_\O-)(f):=1_M\otimes^{ll} f:\quad M\otimes_\O X\to M\otimes_\O X'$$ for $X\stackrel{f}{\longrightarrow}X'$ in category $\mathcal{LO}$ and denote the similar notations for  other tensor functors.

\subsection{{Adjoint functor theorem}}

It turns out that the two functors $$(M\otimes^{ll}_\O-, \quad \Hom_\mathcal{LO}(M,-))$$ constitute an \textbf{adjoint pair}.

\begin{mydef}
	Let $\mathcal{C}$ and $ \mathcal{D}$ be two non-associative categories.
	Let $F: \mathcal{C}\to \mathcal{D}$ and $G: \mathcal{D}\to \mathcal{C}$ be covariant functors. The
	ordered pair $(F, G)$ is an \textbf{adjoint pair} if, for each $C\in \obj{C}$ and $D\in \obj{D}$,
	there are bijections
	$$\tau_{C,D}: w-\Hom_{\mathcal{D}}(FC,D)\to w-\Hom_{\mathcal{C}}(C,GD)$$
	that are natural transformations in ${C}$ and in ${D}$.
\end{mydef}
We get the adjoint functor theorem for $\O$-bimodules as follows:
\begin{thm}
	Let $M$ be an  $\O$-bimodule. Then
	$(M\otimes^{ll}_\O-, \Hom_\mathcal{LO}(M,-))$ is an {adjoint pair}.
	
	More precisely, we have an isomorphism of $\O$-bimodules for any $\O$-bimodules $X,\,Y$:
	$$\tau_{X,Y}:\Hom_\mathcal{LO}(M\otimes_\O X,Y)\to \Hom_\mathcal{LO}( X,\Hom_\mathcal{LO}(M,Y)).$$
	Moreover, the following two diagrams commute
	for all $f : X'\to X$ in $\mathcal{LO}$ and $g : Y\to Y'$ in $\mathcal{LO}$:
	$$\begin{array}[c]{ccc}
	\Hom_\mathcal{LO}(M\otimes_\O X,Y)&\stackrel{\tau_{X,Y}}{\xlongrightarrow[]{\hskip1cm}}&\Hom_\mathcal{LO}( X,\Hom_\mathcal{LO}(M,Y))\\
	\Biggm\downarrow\scriptstyle{(1_M\otimes^{ll} f)^{*_l}}&&\Biggm\downarrow\scriptstyle{ f^{*_l}}\\
	\Hom_\mathcal{LO}(M\otimes_\O X',Y)&\stackrel{\tau_{X',Y}}{\xlongrightarrow[]{\hskip1cm}}&\Hom_\mathcal{LO}( X',\Hom_\mathcal{LO}(M,Y));
	\end{array}$$
	
	$$\begin{array}[c]{ccc}
	\Hom_\mathcal{LO}(M\otimes_\O X,Y)&\stackrel{\tau_{X,Y}}{\xlongrightarrow[]{\hskip1cm}}&\Hom_\mathcal{LO}( X,\Hom_\mathcal{LO}(M,Y))\\
	\Biggm\downarrow\scriptstyle{g_{*_l}}&&\Biggm\downarrow\scriptstyle{ (\Hom_\mathcal{LO}(M,g)_{*_l}}\\
	\Hom_\mathcal{LO}(M\otimes_\O X,Y')&\stackrel{\tau_{X,Y'}}{\xlongrightarrow[]{\hskip1cm}}&\Hom_\mathcal{LO}( X,\Hom_\mathcal{LO}(M,Y')).
	\end{array}$$
\end{thm}

\begin{proof}
	We first define the map $\tau_{X,Y}$ by extension.
	For any  $\alpha\in 	\Hom_\mathcal{LO}(M\otimes_\O X,Y)$, and $x\in \re X$, we define $$\fx{m}{\tau_{X,Y}(\alpha)(x)}:=\alpha(m\otimes_\O x)$$ for all  $m\in \re M$.
	Then by the left \almost extension defined in Lemma  \ref{lem:ext}, we get $$\tau_{X,Y}(\alpha)(x)\in \Hom_\mathcal{LO}(M,Y)$$ for all $x\in \re X$. Then by the left \almost extension again, we obtain $$\tau_{X,Y}(\alpha)\in \Hom_\mathcal{LO}( X,\Hom_\mathcal{LO}(M,Y)).$$
	
	By construction, if $m\in \re M$, then for all $x\in X$ we have
	\begin{eqnarray}\label{eqpf:tau}
	\fx{m}{\tau_{X,Y}(\alpha)(x)}=\alpha(m\otimes_\O x).
	\end{eqnarray}
	Indeed, if we write $x=\sum_{i=0}^7e_ix_i$ as usual, then by left \almost extension in construction, for any $m\in \re M$ we have
	\begin{eqnarray*}
	\fx{m}{\tau_{X,Y}(\alpha)(x)}&=&\sum_{i=0}^7\fx{m}{e_i\odot(\tau_{X,Y}(\alpha)(x_i))}\\
	&=&\sum_{i=0}^7\fx{e_im}{\tau_{X,Y}(\alpha)(x_i)}\\
		&=&\sum_{i=0}^7e_i\fx{m}{\tau_{X,Y}(\alpha)(x_i)}\\
	&=&\sum_{i=0}^7e_i\alpha(m\otimes_\O x_i).
	\end{eqnarray*}
	And it follows from Lemma \ref{lem:tf} that
	$$\alpha(m\otimes_\O x)=\sum_{i=0}^7\alpha(m\otimes_\O e_ix_i)=\sum_{i=0}^7e_i\alpha(m\otimes_\O x_i).$$
	This proves \eqref{eqpf:tau}.

	We come to show that $\tau_{X,Y}$ is $\O$-linear. Bearing Corollary \ref{lem: f(huaa M)=0 yields f=0} in mind,  for all $r\in \O$ and  $\alpha \in \Hom_\mathcal{LO}(M\otimes_\O X,Y)$ we get
	\begin{align*}
	\tau_{X,Y}(r\odot\alpha)=r\odot \tau_{X,Y}(\alpha)&\iff\fx{x}{\tau_{X,Y}(r\odot \alpha)}=\fx{x}{r\odot \tau_{X,Y}(\alpha)},\quad &\text{for all } x\in \re X;\\
	&\iff\tau_{X,Y}(r\odot \alpha)(x)=\tau_{X,Y}(\alpha)(xr),\quad &\text{for all } x\in \re X;\\
	&\iff \fx{m}{\tau_{X,Y}(r\odot \alpha)(x)}=\fx{m}{\tau_{X,Y}(\alpha)(xr)},\quad &\text{for all } x\in \re X, \text{for all } m\in \re M;\\
	&\stackrel{\eqref{eqpf:tau}}{\iff} \fx{m\otimes_\O x}{r\odot \alpha}=\fx{m\otimes_\O xr}{\alpha},\quad &\text{for all } x\in \re X, \text{for all } m\in \re M;\\
	&\iff \fx{(m\otimes_\O x)r}{\alpha}=\fx{m\otimes_\O xr}{\alpha},\quad &\text{for all } x\in \re X, \text{for all } m\in \re M.
	\end{align*}
	Then \eqref{eq:tf2} implies that $\tau_{X,Y}$ is $\O$-linear.	
	It is trivial to see that $\tau_{X,Y}$ is a bijection. Hence we conclude that $\tau_{X,Y}$ is an isomorphism of $\O$-bimodules as desired.
	
	Similarly, to show
	$$f^{*_l}\circ \tau_{X,Y}=\tau_{X',Y}\circ (1_M\otimes^{ll} f)^{*_l},$$
	it suffices to show
	\begin{eqnarray}\label{eqpf:taualph}
	\fx{m}{\fx{x'}{f^{*_l} (\tau_{X,Y}(\alpha))}}=\fx{m}{\fx{x'}{\tau_{X',Y} ((1_M\otimes^{ll} f)^{*_l}(\alpha))}}
	\end{eqnarray}
	for all $\alpha\in  \Hom_\mathcal{LO}(M\otimes_\O X,Y)$ and all $x'\in X'$, $m\in \re M$.
	By direct calculations, we have
	\begin{align*}
	\fx{m}{\fx{x'}{f^{*_l} (\tau_{X,Y}(\alpha))}}&=\fx{m}{\fx{x'}{ \tau_{X,Y}(\alpha)\circledcirc f}}
	\\
	&=\fx{m}{(\tau_{X,Y}(\alpha) (f(x')}&\text{since $x'\in \re X'$}\\
	&=\alpha(m\otimes_\O f(x'))&\text{by \eqref{eqpf:tau}.}
	\end{align*}
	Since $m\otimes_\O x'\in \re (M\otimes_\O X')$, we have
	\begin{align*}
	\fx{m}{\fx{x'}{\tau_{X',Y} ((1_M\otimes^{ll} f)^{*_l}(\alpha))}}&=\fx{m}{\fx{x'}{\tau_{X',Y}(\alpha\circledcirc (1_M\otimes^{ll} f)) }}\\
	&=(\alpha\circledcirc (1_M\otimes^{ll} f)) (m\otimes_\O x')\\
	&=\alpha(m\otimes_\O f(x')).
	\end{align*}
	This proves \eqref{eqpf:taualph}. 
\end{proof}

\subsection{Enveloping categories} Now we introduce a new notion called 
the enveloping category. The point is that it becomes a bridge between
 the non-associative category and the  associative category.
To do this, we have to focus our attention on  additive categories.

\begin{mydef}
 A non-associative additive category is a non-associative category $\mathcal{C}$ satisfying the following
 axioms:
 \begin{enumerate}
  \item  Every set $\operatorname{Hom}_{\mathcal{C}}(X, Y)$  is equipped with a structure of an abelian group (written additively) such that composition of morphisms is biadditive with respect to this structure.
  \item  There exists a zero object  $0 \in \obj{\mathcal{C}}$  such that  $\operatorname{Hom}_{\mathcal{C}}(0,0)  = 0$.
  \item  (Existence of direct sums.) For any objects  $X_{1}, X_{2} \in \obj{\mathcal{C}}$  there exists an object  $Y \in \obj{\mathcal{C}} $ and morphisms  $p_{1}: Y \rightarrow X_{1}$, $p_{2}: Y \rightarrow X_{2}$, $i_{1}: X_{1} \rightarrow Y$,
  $i_{2}: X_{2} \rightarrow Y$  such that  $p_{1} i_{1}  = 1_{X_{1}}$, $p_{2} i_{2}  = 1_{X_{2}}$,  and  $i_{1} p_{1}+i_{2} p_{2}  = 1_{Y}$.
 \end{enumerate}
\end{mydef}

\begin{mydef}
Let  $\mathcal{C}$ be a non-associative additive  category.
\begin{enumerate}
	\item 	$(\mathcal{D}, \iota)$ is called an \textbf{associative cover category} of $\mathcal{C}$ if $\mathcal{D}$ is an associative additive category and  $\iota:\mathcal{C}\to \mathcal{D}$ is a faithful weak additive functor.
	\item An associative cover category $(\widetilde{\mathcal{C}},\iota)$ is called the \textbf{enveloping category} of $\mathcal{C}$ if for any associative cover category $({\mathcal{D}},\eta)$, there exists  a unique additive functor  $T:\widetilde{\mathcal{C}} \to \mathcal{D}$ such that the following diagram
	commutes:
	\begin{center}
		\usetikzlibrary{matrix,arrows}
		\begin{tikzpicture}[description/.style={fill=white,inner sep=2pt}]
		\matrix (m) [matrix of math nodes, row sep=3em,
		column sep=2.5em, text height=1.5ex, text depth=0.25ex]
		{\widetilde{\mathcal{C}} & &  \mathcal{D}  \\
			&\mathcal{C} & \\ };
		\path[->,font=\scriptsize]
		(m-1-1) edge node[auto] {$\exists !\  T $} (m-1-3)
		(m-2-2) edge node[auto] {$ \iota $} (m-1-1)
		edge node[auto] {$\eta $} (m-1-3);
		\end{tikzpicture}
	\end{center}
	
\end{enumerate}

\end{mydef}

	\begin{eg} We now provide some typical examples of enveloping categories.
		\begin{enumerate}
			\item If $\mathcal{C}$ is an associative additive  category, then one can check that  ($\mathcal{C}$, $1_\mathcal{C}$) is the enveloping category of $\mathcal{C}$.
			\item The non-associative category $\mathcal{LO}$ takes
 $(\O\text{-}\textbf{Mod}_\R,\iota)$ as its enveloping category.
Here we recall that the associative category $\O\text{-}\textbf{Mod}_\R$ and the weak functor $\iota:\mathcal{LO}\to \O\text{-}\textbf{Mod}_\R$  are defined in 			Example \ref{eg:o Mod R}. We delay its proof to  Theorem \ref{thm:env cat of LO}.
		\item
	If  $\algma$ be a unital non-associative algebra, then it takes  $(\mathfrak{M}(\algma),\iota)$ as its enveloping category, where
$\mathfrak{M}(\algma)$ is  the  associative  subalgebra generated by
	 left multiplications of $\algma$ and $\iota$ is defined as
	 $$\iota:\algma\to \mathfrak{M}(\algma),\quad x\mapsto L_x.$$
Here   $\algma$ has been viewed as a specific sort of non-associative additive category with a single object, whose   morphisms are
  represented by  the  elements of $\algma$
  and whose composition is the multiplication of $\algma$. Similarly,
	  $\mathfrak{M}(\algma)$ is interpreted as an associative category.
\end{enumerate}
\end{eg}
\begin{thm}\label{thm:env cat of LO}
	The enveloping category of $\mathcal{LO}$  is $(\O\text{-}\textbf{Mod}_\R,\iota)$ and so is $\mathcal{RO}$.
\end{thm}
To prove this theorem, we need   some lemmas.
Let $M$ be an $\O$-bimodule  and $\alpha\in \End_\R(\O)$. Then there associates   a map $\alpha_M\in \End_\R(M)$ defined by
$$\alpha_M(px)=\alpha(p)x$$ for all $p\in \O$ and $x\in \re M$.

\begin{lemma}\label{lem:f apl=alp f}
	Let $f\in \Hom_\O(M,M')$ and $\alpha\in\End_\R(\O) $. Then
	\begin{eqnarray}
	f\circ \alpha_M=\alpha_{M'}\circ f.
	\end{eqnarray}
\end{lemma}
\begin{proof}
	This follows from direct calculations.
	For any $x=\sum_{i=0}^7e_ix_i\in M$,
	\begin{eqnarray*}
		(f\circ \alpha_M)(x)&=&\sum_{i=0}^7f(\alpha(e_i)x_i)\\
		&=&\sum_{i=0}^7\alpha(e_i)f(x_i)\\
		&=&\sum_{i=0}^7\alpha_{M'}(e_if(x_i))\\
			&=&\sum_{i=0}^7\alpha_{M'}(f(e_ix_i))\\
				&=&(\alpha_{M'}\circ f)(x).
	\end{eqnarray*}
This proves the lemma.
\end{proof}

\begin{lemma}
	For any $f\in \Hom_\R(M,M')$, there is a  unique decomposition
	\begin{eqnarray}
	f=\sum_{i,j=0}^7f_{ij}\circ{\alpha^{ij}}_M
	\end{eqnarray}
for which  $f_{ij}\in \Hom_\O(M,M')$,  $\alpha^{ij}\in\End_\R(\O)$, and  $$\alpha^{ij}(e_k)=\delta_{jk}e_i.$$
\end{lemma}
\begin{proof}
	Let $f(x)=\sum_{i=0}^7e_if_i(x)$, where $f_i(x)\in \re(M')$.
	We define $$f_{ij}(x)=f_i(e_jx)$$ for all $x\in \re M$ and then we obtain  $f_{ij}\in \Hom_\O(M,M')$ by para-linear extension. One can check that for all $x=\sum_{i=0}^7e_ix_i\in M$, we have
	\begin{eqnarray*}
	\sum_{i,j=0}^7(f_{ij}\circ{\alpha^{ij}}_M)(x)&=&\sum_{i,j,k=0}^7f_{ij}({\alpha^{ij}}_M(e_kx_k))\\
	&=&\sum_{i,j,k=0}^7f_{ij}({\alpha^{ij}}(e_k)x_k)\\
	&=&\sum_{i,j=0}^7f_{ij}(e_ix_j)\\
	&=&\sum_{i,j=0}^7e_if_{ij}(x_j)\\
	&=&\sum_{i,j=0}^7e_if_{i}(e_jx_j)\\
	&=&\sum_{i=0}^7e_if_{i}(x)\\
		&=&f(x).
	\end{eqnarray*}
The uniqueness is trivial.
\end{proof}

Now we come to prove Theorem \ref{thm:env cat of LO}.

\begin{proof}[Proof of Theorem \ref{thm:env cat of LO}.]
	Let $(\mathcal{D}, \eta)$ be an associative cover category of $\mathcal{LO}$. We need to define a  functor $T:\O\text{-}\textbf{Mod}_\R\to \mathcal{D}$.
	Note that $\End_\R(\O)$ can be generated by $\{R_p\mid p\in \O\}$ algebraically. We define $$T(R_p)=\eta(R_p)$$ for any $R_p\in \End_{\mathcal{LO}}(\O)$ and extend it to $\End_\R(\O)$ by preserving the composition and  linear combination. This induces the definition of $T(\alpha_M)$ for any $\alpha\in\End_\R(\O) $ and $\O$-bimodule $M$. We thus define
	\begin{align*}
	T:\qquad\qquad\O\text{-}\textbf{Mod}_\R\qquad&\xlongrightarrow[]{\hskip1cm}  \qquad\mathcal{D}\\
	M\qquad&\shortmid\!\xlongrightarrow[]{\hskip1cm}  \qquad \eta(M)\\
	f=\sum_{i,j=0}^7f_{ij}\circ{\alpha^{ij}}_M\in \Hom_{\R}(M,M')\qquad&\shortmid\!\xlongrightarrow[]{\hskip1cm}  \qquad T(f):=\sum_{i,j=0}^7\eta({f_{ij}})T({\alpha^{ij}}_M).
	\end{align*}
	We first claim that $$T(f)T( \alpha_M)=T(\alpha_{M'})T(f)$$ for all $f\in \Hom_{\O}(M,M')$ and $\alpha\in \End_\R(\O)$.
	\bfs\ $\alpha_M=R_{p}$ for some $p\in \O$. Since $\eta$ is a weak functor,  it follows that
	$$T(f)T( \alpha_M)=\eta(f)\eta(R_{p})=\eta(f\circ R_p)=\eta( R_p\circ f)= \eta( R_p)\eta(f).$$ This proves the claim.
	For any $$M\stackrel{f}{\longrightarrow}M'\stackrel{g}{\longrightarrow}M''$$
	in $\O\text{-}\textbf{Mod}_\R$, it follows from Lemma \ref{lem:f apl=alp f} that
	\begin{eqnarray*}
	T(g\circ f)&=&T(g_{mn}\circ {\alpha^{mn}}_{M'}\circ f_{ij}\circ {\alpha^{ij}}_M)\\
	&=&T(g_{mn}\circ  f_{ij} \circ{\alpha^{mn}}_{M} \circ {\alpha^{ij}}_M)\\
	&=&\eta(g_{mn}\circ  f_{ij} )T({\alpha^{mn}}_{M} \circ {\alpha^{ij}}_M)\\
	&=&T(g_{mn})T  (f_{ij} ) T({\alpha^{mn}}_{M}) T( {\alpha^{ij}}_M)\\
	&=&T(g_{mn}) T({\alpha^{mn}}_{M'})T(f_{ij} )  T( {\alpha^{ij}}_M)\\
&=&T(g)T(f).
	\end{eqnarray*}
The uniqueness is obvious.
\end{proof}
\subsection{Natural associated functor}
For a  given  functor between non-associative categories, it is natural to  ask what is its natural counterpart   functor in associative categories.

\begin{mydef}
	Let  $\mathcal{C}$ and $\mathcal{D}$ be two non-associative additive categories. Assume   $(\mathcal{\widetilde{C}},\iota_1)$ and $ (\mathcal{\widetilde{D}},\iota_2)$ are their enveloping categories, respectively.
	Suppose  $S:\mathcal{\widetilde{C}} \rightarrow \mathcal{\widetilde{D}}$, $T: \mathcal{C} \rightarrow \mathcal{D}$  are covariant  functors. A \textbf{natural lift transformation}  $\tau: S \rightarrow T $ is a one-parameter family of morphisms in  $\mathcal{\widetilde{D}}$, 	
	$$\tau = \left(\tau_{X}: S (\iota_1(X)) \rightarrow \iota_2(T (X))\right)_{X \in \obj{\mathcal{C}}}$$		making the following diagram
	commute in $\mathcal{\widetilde{D}}$:
	$$\begin{array}[c]{ccc}
	S(\iota_1(X))\quad&\stackrel{\tau_{X}}{\xlongrightarrow[]{\hskip1cm}}&\iota_2(T (X))\\
	\, \Biggm\downarrow\scriptstyle{S(\iota_1(f))}&&\Biggm\downarrow\scriptstyle{\iota_2( T(f))}\\
	S(\iota_1(Y))\quad&\stackrel{\tau_{Y}}{\xlongrightarrow[]{\hskip1cm}}&\iota_2(T (Y)).
	\end{array}$$
	for all  $ X \stackrel{f}{\longrightarrow} Y $ in  $\mathcal{C}$. We call $S$ a natural lift of $T$ in this case.
	
	A \textbf{natural extension transformation}  $\eta: T \rightarrow S $ is a one-parameter family of morphisms in  $\mathcal{\widetilde{D}}$, 	
	$$\eta = \left(\eta_{X}: \iota_2(T (X)) \rightarrow 	S(\iota_1(X))\right)_{X \in \obj{\mathcal{C}}}$$		making the following diagram
	commute in $\mathcal{\widetilde{D}}$:
	$$\begin{array}[c]{ccc}
	\iota_2(T (X))\quad&\stackrel{\eta_{X}}{\xlongrightarrow[]{\hskip1cm}}&	S(\iota_1(X))\\
	\, \Biggm\downarrow\scriptstyle{\iota_2(T(f))}&&\Biggm\downarrow\scriptstyle{ S(\iota_1(f))}\\
	\iota_2(T (Y))\quad&\stackrel{\eta_{Y}}{\xlongrightarrow[]{\hskip1cm}}&	S(\iota_1(Y)).
	\end{array}$$
	for all  $ X \stackrel{f}{\longrightarrow} Y $ in  $\mathcal{C}$. We call $S$ a natural extension of $T$ in this case.
	
Let $\tau$ be a natural lift transformation.	If each $\tau_X$ is an isomorphism  then $\tau$ is called a \textbf{natural lift isomorphism}.
We can define  \textbf{natural extension isomorphism}, similarly.  In these two cases,   $S$ is called a \textbf{natural associated functor} of $T$.

	Natural lift or extended transformation between contravariant  functors can be defined similarly.
\end{mydef}

\begin{rem}
	This definition is also a generalization of the notion of natural transformations in the associative case. It can be seen that the notion of natural extension transformation is the dual concept of the notion of natural lift transformation. The typical examples are given in Theorems \ref{thm:lif nat}
and \ref{thm:ext nat}.	
\end{rem}

Recall  that we have a $\Hom$  functor between non-associative categories
$$\Hom_\mathcal{LO}(-,M):\mathcal{LO}\stackrel{}{\xlongrightarrow[]{\hskip1cm}}\mathcal{RO}$$
for any $\O$-bimodule $M$.
Note that their enveloping  categories are both the associative category  $\O\text{-}\textbf{Mod}_\R$. We have another $\Hom$  functor $\Hom_{\O\text{-}\textbf{Mod}_\R}(-,\re M)$, and  abbreviated by $$\Hom_\R(-, \re M):\O\text{-}\textbf{Mod}_\R\stackrel{}{\xlongrightarrow[]{\hskip1cm}}\O\text{-}\textbf{Mod}_\R,$$ for any given $\O$-bimodule $M$.
We remark that $\Hom_\R(X,\re M)$ is an $\O$-bimodule via the multiplication
$$(pf)(x):=f(px),\quad (fp)(x):=f(xp)$$ for all $p\in \O$ and $x\in X$.
Similar discussion is also  valid for Hom functor $$\Hom_\R( \re M,-):\O\text{-}\textbf{Mod}_\R\stackrel{}{\xlongrightarrow[]{\hskip1cm}}\O\text{-}\textbf{Mod}_\R.$$
It turns out that the bijection $l\text{-}\lif$ defined in subsection \ref{subsec:almost linear} is a natural associated isomorphism: $$l\text{-}\lif:\Hom_\R(-,\re M)\simeq \Hom_\mathcal{LO}(-,M) .$$

\begin{thm}\label{thm:lif nat}
	For any $\O$-bimodule $M$, the left \almost linear lift map is a natural lift isomorphism of contravariant functors
	$$l\text{-}\lif:\Hom_\R(-,\re M)\simeq \Hom_\mathcal{LO}(-,M) .$$
 	
\end{thm}
\begin{proof}
	We need to verify that the following diagram commutes in $\O\text{-}\textbf{Mod}_\R$ for all  $f: Y \rightarrow X $ in  $\mathcal{LO}$:
	$$\begin{array}[c]{ccc}
	\Hom_\R(X,\re M)&\stackrel{l\text{-}\lif\!_{X}}{\xlongrightarrow[]{\hskip1cm}}&\Hom_\mathcal{LO}(X,M)\\
	\Biggm\downarrow\scriptstyle{f^{*_\R}}&&\Biggm\downarrow\scriptstyle{ f^{*_l}}\\
	\Hom_\R(Y,\re M)&\stackrel{l\text{-}\lif\!_{Y}}{\xlongrightarrow[]{\hskip1cm}}&\Hom_\mathcal{LO}(Y,M).
	\end{array}$$
	It suffices to show that
	\begin{eqnarray}\label{eqpf:lif is iso}
	l\text{-}\lif(f^{*_\R}(g))=f^{*_l}(l\text{-}\lif g)
	\end{eqnarray}
	for any $g\in \Hom_\R(X,\re M)$.
	In fact, for all $x\in X$, we have
	\begin{eqnarray*}
		\re \Big(\big(l\text{-}\lif  f^{*_\R}(g) \big)\ (x)\Big)&=&	\re \big(l\text{-}\lif(g\circ f)\ (x)\big)\\
		&=&g\circ f(x)
	\end{eqnarray*}
	and
	\begin{eqnarray*}
		\re \Big(\big (f^{*_l}(l\text{-}\lif g)\big)\ (x)\Big)&=&	\re\Big( \big((l\text{-}\lif g)\circledcirc f\big)\  (x)\Big)\\
		&=&	\re \Big(\big((l\text{-}\lif g)\circ f\big)(x)\Big)\\
		&=&	\re \Big((l\text{-}\lif g)\  \big( f(x)\big)\Big)\\
		&=&g\circ f(x).
	\end{eqnarray*}
	Hence \eqref{eqpf:lif is iso} follows from Corollary \ref{cor: A_p(x,f)=0 and f(px)=pf(x)}.
	Recall that $l\text{-}\lif$ is a real linear  bijection, i.e., an isomorphism in $\O\text{-}\textbf{Mod}_\R$. Hence the left \almost linear lift	$l\text{-}\lif$
	is a natural isomorphism.
\end{proof}

For the \almost linear extension map, we have  a similar result.
\begin{thm}\label{thm:ext nat}
	For any $\O$-bimodule $M$, the left \almost linear  extension
	is  a natural extension isomorphism of  functors$$l\text{-}\ext: \Hom_\R(\re M,-)\simeq \Hom_\mathcal{LO}(M,-) .$$
\end{thm}
\begin{proof}
	We need to verify that the following diagram commutes in $\O\text{-}\textbf{Mod}_\R$ for all  $f: X \rightarrow Y $ in  $\mathcal{LO}$:
	$$\begin{array}[c]{ccc}
	\Hom_\R(\re M,X)&\stackrel{l\text{-}\ext\!_{X}}{\xlongrightarrow[]{\hskip1cm}}&\Hom_\mathcal{LO}(M,X)\\
	\Biggm\downarrow\scriptstyle{f_{*_\R}}&&\Biggm\downarrow\scriptstyle{ f_{*_l}}\\
	\Hom_\R(\re M,Y)&\stackrel{l\text{-}\ext\!_{Y}}{\xlongrightarrow[]{\hskip1cm}}&\Hom_\mathcal{LO}(M,Y).
	\end{array}$$
	It suffices to show that
	\begin{eqnarray}\label{eqpf:ext is iso}
	l\text{-}\ext(f_{*_\R}(g))=f_{*_l}(l\text{-}\ext g)
	\end{eqnarray}
	for any $g\in \Hom_\R(\re M,X)$.
	In fact, let $x\in \re M$, by the definition of left \almost linear extension, we have
	\begin{eqnarray*}
		l\text{-}\ext(f_{*_\R}(g))(x)&=&l\text{-}\ext(f\circ g)(x)\\
		&=&f\circ g(x)
	\end{eqnarray*}
	and
	\begin{eqnarray*}
		f_{*_l}(l\text{-}\ext g)(x)&=&f\circledcirc (l\text{-}\ext g)\,(x)\\
		&=&f\circ (l\text{-}\ext g)\,(x)\\
		&=&f\circ g(x).
	\end{eqnarray*}
	Hence \eqref{eqpf:lif is iso} follows from Corollary \ref{lem: f(huaa M)=0 yields f=0}.
	Since $l\text{-}\ext$ is a  real linear  bijection, i.e., an isomorphism in $\O\text{-}\textbf{Mod}_\R$,  it follows that the left \almost linear extension	$l\text{-}\ext$
	is a natural isomorphism.
\end{proof}
\begin{rem}
	In view of Theorems \ref{thm:lif nat} and \ref{thm:ext nat}, we can say that the counterpart functor of $\Hom_\mathcal{LO}(-,M)$ in associative case  is the usual Hom functor $$\Hom_\R(-, \re M):\O\text{-}\textbf{Mod}_\R\stackrel{}{\xlongrightarrow[]{\hskip1cm}}\O\text{-}\textbf{Mod}_\R$$
	and  the counterpart functor of $\Hom_\mathcal{LO}(M,-)$ in associative case  is the usual Hom functor $$\Hom_\R(\re M,-):\O\text{-}\textbf{Mod}_\R\stackrel{}{\xlongrightarrow[]{\hskip1cm}}\O\text{-}\textbf{Mod}_\R.$$
	Similar results  hold for Hom functor $\Hom_\mathcal{RO}(M,-)$ and  $\Hom_\mathcal{RO}(-,M)$. Hence many properties of $\Hom_\R(-, \re M)$ can be translated into properties of $\Hom_\mathcal{LO}(-,M)$.

\end{rem}

Similarly, we can show that the natural associated functor of $	M\otimes^{ll}_\O-$   is the usual tensor functor
$$\re M\otimes_\R-:\O\text{-}\textbf{Mod}_\R\stackrel{}{\xlongrightarrow[]{\hskip1cm}}\O\text{-}\textbf{Mod}_\R.$$The proof is straight forward
and is omitted.

\subsection{Exactness of functors}

In this subsection, we discuss the exactness of functors.
We attribute the issue of exactness   to the associative situation
via the notion of  the enveloping category.
\begin{mydef}
	Let $\mathcal{C}$ be a non-associative additive categories and $(\mathcal{\widetilde{C}},\iota)$ be its enveloping category. A sequence in $\mathcal{C}$ is said to be exact if its image under $\iota$ in $\mathcal{\widetilde{C}}$ is exact. Similar notions can be also defined for short exact sequence, left (right) exact sequence and complex. We can also define left (right) exact functors as usual.
\end{mydef}
\begin{thm}
	Let $M$ be an $\O$-bimodule. The Hom functor $\Hom_\mathcal{LO}(M,-)$ is a left exact functor.
\end{thm}
\begin{proof}
	Let $ 0\longrightarrow X \stackrel{f}{\longrightarrow} X'\stackrel{g}{\longrightarrow} X''$ be an exact sequence in $\mathcal{LO}$. We need to verify that there is an exact	sequence
	$$ 0\longrightarrow \Hom_\mathcal{LO}(M,X) \stackrel{f_{*_l}}{\longrightarrow} \Hom_\mathcal{LO}(M,X')\stackrel{g_{*_l}}{\longrightarrow} \Hom_\mathcal{LO}(M,X'').$$

	We conclude from Theorem \ref{thm:ext nat} that there is an natural associated isomorphism $\tau$ such that the following  diagram commutes in $\O\text{-}\textbf{Mod}_\R$:
	$$\begin{array}[c]{ccc}
	0 \quad&\stackrel{ }{\xlongrightarrow[]{\hskip1cm}}& 0\\
	\Biggm\downarrow\scriptstyle{\ \ }&&\Biggm\downarrow\scriptstyle{ }\\
	\Hom_\R(\re     M,X)&\stackrel{\tau_{X}}{\xlongrightarrow[]{\hskip1cm}}&\Hom_\mathcal{LO}(M,X)\\
	\Biggm\downarrow\scriptstyle{f_{*_\R}}&&\Biggm\downarrow\scriptstyle{ f_{*_l}}\\
	\Hom_\R(\re M,X')&\stackrel{\tau_{X'}}{\xlongrightarrow[]{\hskip1cm}}&\Hom_\mathcal{LO}(M,X')\\
	\Biggm\downarrow\scriptstyle{g_{*_\R}}&&\Biggm\downarrow\scriptstyle{ g_{*_l}}\\
	\Hom_\R(\re M,X'')&\stackrel{\tau_{X''}}{\xlongrightarrow[]{\hskip1cm}}&\Hom_\mathcal{LO}(M,X'').
	\end{array}$$
	By the classical theory, we have an exact	sequence
	$$ 0\longrightarrow \Hom_\R(\re M,X) \stackrel{f_{*_\R}}{\longrightarrow} \Hom_\R(\re M,X')\stackrel{g_{*_\R}}{\longrightarrow} \Hom_\R(\re M,X'').$$ The means that the sequence  $$ 0\longrightarrow \Hom_\mathcal{LO}(M,X) \stackrel{f_{*_l}}{\longrightarrow} \Hom_\mathcal{LO}(M,X')\stackrel{g_{*_l}}{\longrightarrow} \Hom_\mathcal{LO}(M,X'')$$ is   exact.
\end{proof}

An analogous result also holds  for contravariant Hom functor $\Hom_\mathcal{LO}(-,M)$.
By the similar argument, we can show that the tensor functor $	M\otimes^{ll}_\O-$ is a right exact functor.

\section{Concluding remarks}

By taking   octonionic multiplications as a model, we have introduced a new notion of \almost linearity in the non-associative setting.
This notion inspires us to bring many results from the non-associative case to the associative case. For example, we obtain two typical
non-associative categories  of   $\mathcal{LO}$ and $\mathcal{RO}$ and start the study of non-associative categories. This breaks the restrictions on the classical associative compounds.  Some new notions are proposed, which play important roles in  the study of non-associative categories. For example, we introduce the notions of weak functors and enveloping categories. This establishes  a close relation between associative category and non-associative category. These make  us to establish  the Yoneda lemma in the setting of non-associative category and introduce exactness for non-associative category.
We introduce the notion of natural associated functors so that we can understand better functors between non-associative categories.

Although  the theory of homology algebra    is very mature in its associative case   \cite{weibel1994homologicalgebra},  it remains untouched  in the non-associative case.

{It deserves to study further the theory of \almost linear operators
in  $\O$-Hilbert spaces and $\O$-Banach spaces.    Based on para-linear functionals, we can introduce the octonionic weak topology of an  $\O$-Hilbert space.  The Hom functor will enable us to study    dual operators and  reflexive spaces in octonionic functional analysis. By considering a para-bilinear multiplication in an $\O$-bimodule, we can introduce the notion of $\O$-algebras. This provides foundations for the  study of octonionic $C^*$-algebras.  The Gelfand-Naimark theorem for the  octonionic $C^*$-algebras may open  new horizons in the study   of   non-associative geometry.}

{The $\O$-\almost linearity over $\O$-bimodules may also provide  a tool in the study of Riemannian geometry.  The theory of octonion bundles \cite{Grigorian2017octonionbundles} as well as loop bundles  \cite{Grigorian2020loopbundles} have been developed by Grigorian recently for the study of $G_2$-manifold. It is quite natural to study \almost linear
  connections instead of
 linear connections in $G_2$-manifold. \cite{Leung2010riemgeom} presents a nice review of various geometric structures on Riemannian manifolds and discusses a unified description using normed division algebras. And we remark that the notion of para-linearity is   a unified description for the classical linearity on normed division algebras. We believe that this notion will play an important role in the future.
}

 \bibliographystyle{plain}


\begin{thebibliography}{10}
	
	
	\bibitem{Awodey_2006cat}
	S.~Awodey.
	\newblock {\em {C}ategory {T}heory}, volume 52 of {\em Oxford Logic Guides}.
	\newblock Oxford University Press, Oxford, second edition, 2010.

	
	\bibitem{baez2002octonions}
	J.~C. Baez.
	\newblock The octonions.
	\newblock {\em Bull. Amer. Math. Soc. (N.S.)}, 39(2):145--205, 2002.
	
	\bibitem{bryant2003some}
	R.~L. Bryant.
	\newblock Some remarks on {$G_2$}-structures.
	\newblock In {\em Proceedings of {G}\"{o}kova {G}eometry-{T}opology
		{C}onference 2005}, pages 75--109. G\"{o}kova Geometry/Topology Conference
	(GGT), G\"{o}kova, 2006.
	
	\bibitem{colombo2011noncomfunctcalculus}
F.~Colombo  and  I.~Sabadini and D.~C. Struppa.
\newblock {\em {N}oncommutative functional calculus}, volume 289 of {\em Progress in Mathematics}.
\newblock Birkh\"{a}user/Springer Basel AG, Basel,  2011. Theory and applications of slice hyperholomorphic functions


	\bibitem{Colombo2008funcalculus}
F.~Colombo  and  I.~Sabadini and D.~C. Struppa.
\newblock {A} new functional calculus for noncommuting operators
\newblock {\em J. Funct. Anal.}, 254(8):2255--2274, 2008.












	
	\bibitem{ghiloni2013slicefct}
 R.~Ghiloni and V.~Moretti and A.~Perotti.
\newblock Continuous slice functional calculus in quaternionic {H}ilbert
spaces.
\newblock {\em Rev. Math. Phys.}, 25(4):1350006, 83, 2013.

	\bibitem{Ghiloni2018semigp}
R.~Ghiloni and R.~Vincenzo.
\newblock Slice regular semigroups.
\newblock {\em Trans. Amer. Math. Soc.}, 370(7):4993--5032, 2018.





	
	\bibitem{goldstine1964hilbert}
	H.~H. Goldstine and L.~P. Horwitz.
	\newblock Hilbert space with non-associative scalars. {I}.
	\newblock {\em Math. Ann.}, 154:1--27, 1964.
	
	\bibitem{Grigorian2017octonionbundles}
	S.~Grigorian.
	\newblock {$G_2$}-structures and octonion bundles.
	\newblock {\em Adv. Math.}, 308:142--207, 2017.
	
	\bibitem{Grigorian2020loopbundles}
	S.~Grigorian.
	\newblock Smooth loops and loop bundles.
	\newblock {\em   arXiv:2008.08120}, 2020.
	
	\bibitem{horwitz1993QHilbertmod}
	L.~P. Horwitz and A.~Razon.
	\newblock Tensor product of quaternion {H}ilbert modules.
	\newblock In {\em Classical and quantum systems ({G}oslar, 1991)}, pages
	266--268. World Sci. Publ., River Edge, NJ, 1993.
	
	
	\bibitem{huo2021leftmod}
	Q.~Huo, Y.~Li, and G.~Ren.
	\newblock Classification of {L}eft {O}ctonionic {M}odules.
	\newblock {\em Adv. Appl. Clifford Algebr.}, 31(1):Paper No. 11, 2021.
	
	\bibitem{huoqinghai2021Riesz}
	Q.~Huo and G.~Ren.
	\newblock {P}ara-linearity as the nonassociative counterpart of linearity.
	\newblock {\em arXiv:2107.08162,} 2021.
	
	\bibitem{jacobson1954structure}
	N.~Jacobson.
	\newblock Structure of alternative and {J}ordan bimodules.
	\newblock {\em Osaka Math. J.}, 6:1--71, 1954.

	
	
	
		\bibitem{Leung2010riemgeom}
N. ~C.	Leung,
	\newblock Riemannian geometry over different normed division algebras.
	\newblock {\em J. Differential Geom.}, 61(2):289-333, 2002.
	
	\bibitem{MacLane1998Categories}
	S.~Mac~Lane.
	\newblock {\em Categories for the working mathematician}, volume~5 of {\em
		Graduate Texts in Mathematics}.
	\newblock Springer-Verlag, New York, second edition,  1998.
	
	\bibitem{ng2007quaternionic}
	C.~K. Ng.
	\newblock On quaternionic functional analysis.
	\newblock {\em Math. Proc. Cambridge Philos. Soc.}, 143(2):391--406, 2007.
	

	
	\bibitem{razon1991projection}
	A. Razon  and L.~P. Horwitz.
	\newblock Projection operators and states in the tensor product of quaternion
	{H}ilbert modules.
	\newblock {\em Acta Appl. Math.}, 24(2):179--194, 1991.
	
	
		\bibitem{razon1992Uniqueness}
	A.~Razon and L.~P. Horwitz.
	\newblock Uniqueness of the scalar product in the tensor product of quaternion
	{H}ilbert modules.
	\newblock {\em J. Math. Phys.}, 33(9):3098--3104, 1992.
	
	\bibitem{Rotman2009homologicalalgebra}
	J.~J. Rotman.
	\newblock {\em An introduction to homological algebra}.
	\newblock Universitext. Springer, New York, second edition, 2009.
	
	\bibitem{Schafer1952repaltalg}
	R.~D. Schafer.
	\newblock Representations of alternative algebras.
	\newblock {\em Trans. Amer. Math. Soc.}, 72:1--17, 1952.
	
	\bibitem{Shestakov2016bimod}
	I.~P. Shestakov and M.~Trushina.
	\newblock Irreducible bimodules over alternative algebras and superalgebras.
	\newblock {\em Trans. Amer. Math. Soc.}, 368(7):4657--4684, 2016.
	
	\bibitem{soffer1983quaternion}
	A.~Soffer and L.~P. Horwitz.
	\newblock {$B^{\ast} $}-algebra representations in a quaternionic {H}ilbert
	module.
	\newblock {\em J. Math. Phys.}, 24(12):2780--2782, 1983.
	
	\bibitem{viswanath1971normal}
	K.~Viswanath.
	\newblock Normal operations on quaternionic {H}ilbert spaces.
	\newblock {\em Trans. Amer. Math. Soc.}, 162:337--350, 1971.
	
	
	
	
	
	\bibitem{weibel1994homologicalgebra}
	C.~A. Weibel.
	\newblock {\em An introduction to homological algebra}, volume 38
	\newblock Cambridge University Press, Cambridge. 1994
	
	
	\bibitem{zhevlakov1982Rings}
	K.~A. Zhevlakov, A.~M. Slinko, I.~P. Shestakov, and A.~I. Shirshov.
	\newblock {\em Rings that are nearly associative}, volume 104 of {\em Pure and
		Applied Mathematics}.
	\newblock Academic Press, Inc., New
	York-London, 1982.
\bigskip\bigskip\bigskip\bigskip	
\end{thebibliography}

\end{document}